\documentclass[11pt,twoside]{amsart}
\usepackage{latexsym,amssymb,amsmath}

\usepackage[pagebackref]{hyperref}
\usepackage{emlines}
\numberwithin{equation}{section}
\newtheorem{theorem}{Theorem}[section]
\newtheorem{lemma}[theorem]{Lemma}
\newtheorem{proposition}[theorem]{Proposition}
\newtheorem{corollary}[theorem]{Corollary}
\newtheorem{conjecture}[theorem]{Conjecture}

\theoremstyle{definition}
\newtheorem{definition}[theorem]{Definition} 
 
\newtheorem{remark}[theorem]{Remark}
\newtheorem{example}[theorem]{Example}

\newtheorem{question}[theorem]{Question}

\def\Ass{\operatorname{Ass}}
\def\Min{\operatorname{Min}}
\def\height{\operatorname{height}}

\def\C{{\mathcal C}}
\newcommand{\m}{\mathfrak{m}}

\begin{document}

%%%%%%%%%%%%%%%%%%%%%%%%%%%%%%%%%%%%%%%%%%%%%%%%%%%%%%%%%%%%%%%%%%%%%

\title[]{Edge ideals: algebraic and combinatorial properties} 

\author{Susan Morey}
\address{Department of Mathematics \\
Texas State University\\
601 University Drive\\ 
San Marcos, TX 78666.}
\email{morey@txstate.edu}
\thanks{The second author was partially supported by CONACyT 
grant 49251-F and SNI}

\author{Rafael H. Villarreal}
\address{
Departamento de
Matem\'aticas\\
Centro de Investigaci\'on y de Estudios
Avanzados del
IPN\\
Apartado Postal
14--740 \\
07000 Mexico City, D.F.
}
\email{vila@math.cinvestav.mx}


\begin{abstract} Let $\mathcal{C}$ be a clutter and let
$I(\mathcal{C})\subset R$ be its edge ideal. This is a survey paper on
  the algebraic and combinatorial  
properties of $R/I(\mathcal{C})$ and $\mathcal{C}$, respectively. We
  give a criterion to 
estimate the regularity of $R/I(\mathcal{C})$ and apply this criterion
  to give new 
proofs of some formulas for the regularity. If $R/I({\mathcal{C}})$ is sequentially 
Cohen-Macaulay, we present a formula for the regularity of the ideal
of vertex covers of $\mathcal{C}$ and give a formula for the
projective dimension of $R/I(\mathcal{C})$. We also examine the associated
primes of powers of edge ideals, and show that for a graph with a
leaf, these sets form an ascending chain.
\end{abstract}

\maketitle

\section{Introduction}

A {\it clutter\/} $\mathcal{C}$ is a finite ground set $X$ 
together with a family $E$ of subsets of $X$
such that if $f_1, f_2 \in E$, then
$f_1\not\subset f_2$. The ground set $X$ is called the {\em vertex
  set} of $\mathcal{C}$ and $E$  
is called the {\em edge set} of $\mathcal{C}$, denoted by
$V(\mathcal{C})$ 
and $E(\mathcal{C})$  respectively. Clutters are simple hypergraphs
and are sometimes called  
{\em Sperner families} in the literature. We can also think of a clutter 
as the maximal faces of a simplicial complex over a ground set. One 
example of a clutter is a graph with the vertices and edges defined in the 
usual way.

Let $\mathcal{C}$ be a clutter with vertex set $X=\{x_1,\ldots,x_n\}$ and 
with edge set $E(\mathcal{C})$. Permitting an abuse of notation, we
will also denote 
by $x_i$ the $i^{\textup{th}}$ 
variable in the polynomial ring $R=K[x_1,\ldots,x_n]$ over a field $K$. 
The {\it edge ideal\/} of $\mathcal{C}$, denoted by $I(\mathcal{C})$,
is the ideal of $R$ generated by all  
monomials $x_e=\prod_{x_i\in e}x_i$ such 
that $e\in E(\mathcal{C})$. Edge ideals of graphs and clutters were 
introduced in \cite{Vi2} and \cite{Faridi,clutters,Ha-VanTuyl}, respectively. 
The assignment $\mathcal{C} \mapsto I(\mathcal{C})$ establishes a natural 
one-to-one correspondence between the family of clutters and the family of 
square-free monomial ideals. Edge ideals of clutters
are also called {\it facet ideals} \cite{Faridi}. 

This is a survey paper on edge ideals, which includes some new proofs
of known results and some new results. The study of algebraic and
combinatorial  
properties of edge ideals and clutters (e.g., Cohen-Macaulayness,
unmixedness, normality, normally torsion-freeness, shellability,
vertex decomposability, stability of 
associated primes) is of current interest, see
\cite{dochtermann,mfmc,poset,emtander,Faridi,FHVodd,FHV,
reesclu,HaM,herzog-hibi-book,HHTZ,MRV,Woodroofe}
 and the references there. In this paper we will focus on the following algebraic properties: the sequentially 
Cohen-Macaulay property, the stability of associated primes, and the
connection between torsion-freeness and combinatorial problems.    

The numerical invariants of edge ideals 
have attracted a great deal of interest
\cite{barile,bounds,Ha-VanTuyl,kummini,morey,nevo,peeva-stillman,vantuyl,monalg,woodroofe-matchings,zheng}.
 In this paper we focus on
the following invariants: projective dimension, regularity, depth and
Krull dimension. 

We present a few new results on edge ideals. We give a criterion to
estimate the regularity of edge ideals (see
Theorem~\ref{reg-of-families}). We apply this criterion to give new
proofs of some formulas for the regularity of edge ideals (see
Corollary~\ref{chordal-reg}). If $\mathcal{C}$ is a clutter and $R/I({\mathcal{C}})$ is sequentially
Cohen-Macaulay, we present a formula for the regularity of the ideal
of vertex covers of $\mathcal{C}$ (see
Theorem~\ref{reg-vila-morey-scm}) and give a formula for the
projective dimension of $R/I(\mathcal{C})$ (see
Corollary~\ref{vila-morey-pd}). We also give a new class of monomial
ideals for which the sets of associate primes of powers are known to form
ascending chains (Proposition~\ref{leaf}).

For undefined terminology on commutative algebra, edge ideals, graph theory, and the
theory of clutters and hypergraphs we refer to 
\cite{BHer,Eisen,Stanley}, \cite{faridijct,monalg}, \cite{Boll,Har},
\cite{cornu-book,Schr2}, respectively.  
 
\section{Algebraic and combinatorial properties of edge ideals}

Let $\mathcal{C}$ be a clutter with vertex set $X=\{x_1,\ldots,x_n\}$
and let $I=I(\mathcal{C})\subset R$ be its edge ideal.  
A subset $F$ of $X$ is called 
{\it independent\/} or {\it stable\/} if $e\not\subset F$ for any  
$e\in E(\mathcal{C})$. The dual concept of a stable vertex set
is a {\it vertex cover\/}, i.e., a subset $C$ of $X$ is a vertex cover of
$\mathcal{C}$  
if and only if $X\setminus C$ is a stable vertex set. A first hint of
the rich interaction between the combinatorics of $\mathcal{C}$  
and the algebra of $I(\mathcal{C})$ is that the
number of vertices in a minimum vertex cover of $\mathcal{C}$ 
(the {\it covering number\/} $\alpha_0(\mathcal{C})$ of $\mathcal{C}$) coincides with 
 the {\it height\/} ${\rm ht}\, I(\mathcal{C})$ of the ideal
$I(\mathcal{C})$. The number of vertices in a maximum stable
set (the {\it stability number} of $\mathcal{C}$) is denoted by $\beta_0(\mathcal{C})$.
Notice that $n=\alpha_0(\mathcal{C})+\beta_0(\mathcal{C})$.

A less immediate interaction between the two fields comes from passing
to a simplicial complex and relating combinatorial properties of the
complex to algebraic properties of the ideal. 
The Stanley-Reisner complex of
$I(\mathcal{C})$, denoted by $\Delta_\mathcal{C}$, is the simplicial
complex whose faces are the independent vertex sets of $\mathcal{C}$.
The complex $\Delta_\mathcal{C}$ is also called the {\it independence
complex} of $\mathcal{C}$. Recall 
that $\Delta_\mathcal{C}$ is called {\it pure\/} 
if all maximal independent vertex sets of $\mathcal{C}$, with respect
to inclusion, have the same number of
elements. If $\Delta_\mathcal{C}$ is pure (resp. Cohen-Macaulay,
shellable, vertex decomposable), we say that $\mathcal{C}$ is
{\it unmixed\/} (resp. Cohen-Macaulay, shellable, vertex
decomposable). Since minor variations of the definition of
shellability exist in the literature, we state here the definition
used throughout this article. 

\begin{definition}\label{Shellabledefn}\rm 
A simplicial complex $\Delta$ is {\it shellable\/}  if 
the facets (maximal faces) of $\Delta$ can be ordered $F_1,\ldots,F_s$
such that  
for all $1\leq i<j\leq s$, there 
exists some $v\in F_j\setminus F_i$ and some 
$\ell\in \{1,\ldots,j-1\}$ with $F_j\setminus F_\ell= \{v\}$. 
\end{definition}

We are 
interested in determining which families of clutters have the property
that $\Delta_\mathcal{C}$ is pure, Cohen-Macaulay, 
or shellable. These 
properties have been extensively studied, see
\cite{BHer,MRV,plummer-unmixed,plummer-survey,ravindra,ITG,Stanley,monalg,unmixed} 
and the references there.

The above definition of shellable is due to 
Bj\"orner and Wachs \cite{BW} and is usually referred
to as {\it nonpure shellable}, although here we will drop the
adjective ``nonpure''.  Originally, the definition of
shellable also required that 
the simplicial complex be pure, that is, all facets have the same
dimension.  
We will say $\Delta$ is {\it pure shellable} if it also satisfies this
hypothesis. These properties are related to
other important properties \cite{BHer,Stanley,monalg}: 
$$
\begin{array}{ccccccc}
\mbox{pure shellable}&\Rightarrow &\mbox{constructible} &\Rightarrow 
&\mbox{Cohen-Macaulay}&\Leftarrow&\mbox{Gorenstein}. 
\end{array}
$$ 
If a shellable complex is not pure, an implication similar to that
above holds when Cohen-Macaulay is replaced by sequentially Cohen-Macaulay.
\begin{definition}\label{d.seqcm}\rm 
Let $R=K[x_1,\dots,x_n]$. A graded $R$-module $M$ is called 
{\it sequentially Cohen-Macaulay} (over $K$)
if there exists a finite filtration of graded $R$-modules
\[(0) = M_0 \subset M_1 \subset \cdots \subset M_r = M \]
such that each $M_i/M_{i-1}$ is Cohen-Macaulay, and the Krull dimensions of the
quotients are increasing:
\[\dim (M_1/M_0) < \dim (M_2/M_1) < \cdots < \dim (M_r/M_{r-1}).\]
\end{definition}
We call a clutter $\mathcal{C}$ {\it sequentially Cohen-Macaulay\/} 
if $R/I(\mathcal{C})$ is sequentially Cohen-Macaulay. As first shown
by Stanley \cite{Stanley}, 
shellable implies sequentially Cohen-Macaulay. 

A related notion for a
simplicial complex is that of {\it vertex decomposability\/}
\cite{bjorner-topological}. If $\Delta$ is a simplicial complex and
$v$ is a vertex of $\Delta$, then the subcomplex formed by deleting
$v$ is the simplicial complex consisting of the faces of $\Delta$ that
do not contain $v$, and the {\it link} of $v$ is 
$$lk(v)=\{ F \in \Delta | v \not\in F \, {\mbox{\rm and}}\, F\cup\{v\} \in \Delta\}.$$
Suppose $\Delta$  is a (not necessarily pure) simplicial complex. We say that
$\Delta$  is vertex-decomposable if either $\Delta$ is a simplex, or
$\Delta$ contains a vertex $v$ such that both the link of $v$ and the subcomplex formed by deleting $v$ are vertex-decomposable, and
such that every facet of the deletion is a facet of $\Delta$.
 If $\mathcal{C}$ is vertex decomposable,
i.e., $\Delta_\mathcal{C}$ is vertex decomposable, then
$\mathcal{C}$ is shellable and sequentially Cohen-Macaulay
\cite{bjorner-topological,Woodroofe}. Thus, we have:
$$
\begin{array}{ccccccc}
\mbox{vertex decomposable }&\Rightarrow &\mbox{shellable} &\Rightarrow 
&\mbox{sequentially Cohen-Macaulay.}&& 
\end{array}
$$ 

Two additional properties related to the properties above are also of
interest in this area. One is the unmixed property, which is implied
by the Cohen-Macaulay property. The other is balanced. To define
balanced, it is useful to have a matrix that encodes the edges of a graph or
clutter. 
\begin{definition}
Let $f_1,\ldots,f_q$ be the edges of a clutter $\mathcal{C}$. The
{\it incidence matrix\/} or {\it clutter matrix} of 
$\mathcal{C}$ is the $n\times q$ matrix $A=(a_{ij})$ given by 
$a_{ij}=1$ if $x_i\in f_j$ and $a_{ij}=0$ otherwise. We say that
$\mathcal{C}$ is a {\it totally balanced\/} clutter 
(resp. {\it balanced} clutter) if $A$ has no square 
submatrix of order at least $3$ (resp. of odd order) with exactly two $1$'s in 
each row and column. 
\end{definition}

If $G$ is a graph, then $G$ is balanced
if and only if $G$ is bipartite and $G$ is totally balanced if and
only if $G$ is a forest \cite{Schr,Schr2}. 

While the implications between the properties mentioned above are
interesting in their own right, it is useful to identify classes of
ideals that satisfy the various properties. We begin with the
Cohen-Macaulay and unmixed properties. 
There are classifications of the following families in terms of
combinatorial properties of the graph or clutter:

\begin{enumerate} 

\item[($\mathrm{c}_1$)] \cite{ravindra,unmixed} unmixed bipartite graphs,

\item[($\mathrm{c}_2$)] \cite{EV,herzog-hibi} Cohen-Macaulay bipartite graphs,

\item[($\mathrm{c}_3$)] \cite{Vi2} Cohen-Macaulay trees,

\item[($\mathrm{c}_4$)] \cite{faridijct} totally
balanced unmixed clutters,  

\item[($\mathrm{c}_5$)] \cite{MRV} unmixed clutters with the
K\"onig property without cycles of length $3$ or $4$,

\item[($\mathrm{c}_6$)] \cite{MRV} unmixed balanced clutters.   
\end{enumerate}

We now focus on the sequentially Cohen-Macaulay property. 
\begin{proposition}{\cite{FVT2}}\label{only-scm-cycles}
The only
sequentially Cohen-Macaulay cycles are ${C}_3$ and ${C}_5$
\end{proposition}

The next theorem generalizes a result of \cite{EV} (see
$(\mathrm{c}_2)$ above) which
shows that a bipartite graph $G$ is Cohen-Macaulay if and only
if $\Delta_G$ has a pure shelling.

\begin{theorem} \cite{bipartite-scm} \label{SCM=shellable}
Let $G$ be a bipartite graph. Then $G$ is
shellable if and only if $G$ is sequentially Cohen-Macaulay. 
\end{theorem}

Recently Van Tuyl \cite{vantuyl} has shown that
Theorem~\ref{SCM=shellable} 
remains valid if we replace shellable by vertex decomposable.

Additional examples of sequentially Cohen-Macaulay ideals depend on
the chordal structure of the graph.
A graph $G$ is said to be 
{\it chordal\/} if every cycle of $G$ of length at least $4$ has a 
chord. A {\it chord} of a cycle is an edge joining two 
non-adjacent vertices of the cycle. 
Chordal graphs have been extensively studied, and they can be
constructed according to a result of G.~A.~Dirac, see  
\cite{Dirac,hhz-ejc,Toft}. A chordal graph
is called {\it strongly chordal} if 
every cycle $C$ of even length at least six has a chord that divides
$C$ into two odd length paths. A {\it clique\/} of a graph is a set of mutually adjacent 
vertices.  Totally balanced clutters are precisely
the clutters of maximal cliques of strongly chordal graphs 
by a result of Farber \cite{farber}. Faridi \cite{Faridi} introduced
the notion 
of a simplicial forest. In
\cite[Theorem~3.2]{HHTZ} it is shown that $\mathcal{C}$ is the clutter of
the facets of a simplicial forest if and only if $\mathcal{C}$ is a
totally balanced clutter. Additionally, a clutter $\mathcal{C}$ is
called $d$-{\it uniform} if all its edges have size $d$.

\begin{theorem}\label{scm-clutters} Any of the following clutters is
sequentially Cohen-Macaulay{\rm :}
\begin{enumerate}
\item[(a)] \cite{Woodroofe}  graphs with no chordless 
cycles of length other than $3$ or $5$,
\item[(b)] \cite{FVT2} chordal graphs, 
\item[(c)] \cite{hhz-ejc} clutters whose ideal of covers has
  linear quotients $($see Definitions~\ref{def.linearquotients} and
  \ref{alex-dual}$)$,
\item[(d)] \cite{He-VanTuyl} clutters of paths of length $t$
of directed rooted trees,
\item[(e)] \cite{Faridi}  simplicial forests, i.e., totally balanced
clutters, 
\item[(f)] \cite{linearquotients} uniform admissible clutters whose 
covering number is $3$.
\end{enumerate}
\end{theorem}

The clutters of parts (a)--(f) are in fact shellable, and the clutters
of parts (a)--(b) are in fact vertex decomposable, see
\cite{dochtermann,hhz-ejc,vantuyl,bipartite-scm,woodroofe-chordal,Woodroofe}.
The family of graphs in part (b) is contained in the family of graphs
of part (a) because the only induced cycles of a chordal graph are $3$-cycles.

A useful tool in examining invariants related to resolutions comes
from a carefully chosen ordering of the generators.
\begin{definition} \label{def.linearquotients}
A monomial ideal $I$ has {\it linear quotients} if the
monomials that generate 
$I$ can be ordered $g_1,\ldots,g_q$ such that 
for all $1 \le i \le q-1$, $((g_1, \dots, g_i) :
g_{i+1})$ is generated by linear forms.      
\end{definition}

If an edge ideal $I$ is generated in a single
degree and $I$ has linear quotients, then $I$ has a linear resolution
(cf. \cite[Lemma 5.2]{Faridi}). If $I$ is the edge ideal of a graph,
then $I$ has linear quotients if and only if $I$ has a linear
resolution if and only if each power of $I$ has a linear resolution
\cite{hhz-linear}.

Let $G$ be a graph. Given a subset $A \subset V(G)$, by $G \setminus
A$, we mean 
the graph formed from $G$ by deleting all the vertices in $A$, and all
edges incident to a vertex in $A$. 
A graph $G$ is called {\it vertex-critical}
if $ \alpha_0{(G\setminus\{v\})}< \alpha_0{(G)}$ for all $v\in V(G)$. An
{\it edge critical} graph is defined similarly. The final property
introduced in this section is a combinatorial 
decomposition of the vertex set of a graph.

\begin{definition} \cite{berge1} \rm A graph $G$ without isolated
vertices is called a 
$B$-{\it graph\/} if there is a family ${\mathcal G}$ consisting of independent
sets of $G$ such that $V(G)=\bigcup_{C \in{\mathcal G}} C$
and $|C|=\beta_0(G)$ for all $C\in{\mathcal G}$.
\end{definition}

The notion of a $B$-{\it graph} is at the
center of several interesting families of graphs. One has the
following implications for any graph $G$ without isolated
vertices \cite{berge1,monalg}:
$$
\begin{array}{ccccccc}
&&\mbox{edge-critical} &\Longrightarrow&& & \\
%&&( \alpha\mbox{-edge-critical)}& & & & \\ 
&& & &B\mbox{-graph} &\Longrightarrow &\mbox{vertex-critical} \\
%&&&& & &(\tau\mbox{-vertex-critical).} \\ 
%&&\mbox{unmixed}&\Longrightarrow & & & \\ 
\mbox{Cohen-Macaulay}&\Longrightarrow&\mbox{unmixed}&\Longrightarrow & & & 
\end{array}
$$ 
In \cite{berge1} the integer $\alpha_0(G)$ is called the {\it transversal 
number\/} of $G$.  

\begin{theorem} \cite{erdos-gallai,bounds} \label{teo1}
 If\/ $G$ is a $B$-graph, 
then $\beta_0(G)\leq  \alpha_0(G)$. 
\end{theorem}

\section{Invariants of edge ideals: regularity, projective dimension,
depth}\label{algebraic-invariants}

Let $\mathcal{C}$ be a clutter and let $I=I(\mathcal{C})$ be its edge
ideal. In this section we study the regularity, depth, projective
dimension, and Krull dimension of $R/I(\mathcal{C})$. There are several well-known results relating these invariants that will prove useful. We collect some of them here for ease of reference. 

The first result is a basic relation between the dimension and the depth (see for example 
\cite[Proposition~18.2]{Eisen}):
\begin{equation}
{\rm depth}\, R/I(\mathcal{C})\leq\dim R/I(\mathcal{C}).
\end{equation}
The deviation from equality in the above relationship can be
quantified using the projective dimension, as is seen in a formula
discovered by Auslander and Buchsbaum (see  
\cite[Theorem~19.9]{Eisen}): 
\begin{equation}\label{auslander-buchsbaum-formula}
{\rm pd}_R(R/I(\mathcal{C}))+{\rm depth}\, R/I(\mathcal{C})={\rm depth}(R).
\end{equation}
Notice that since in the setting of this article $R$ is a polynomial ring in $n$ variables, ${\rm depth}( R)=n$.

Another invariant of interest also follows from a closer inspection of
a minimal projective resolution of $R/I$. Consider the minimal graded
free resolution of $M=R/I$ as an 
$R$-module: 
\[
{\mathbb F}_\star:\ \ \ 0\rightarrow 
\bigoplus_{j}R(-j)^{b_{gj}}
\stackrel{}{\rightarrow} \cdots 
\rightarrow\bigoplus_{j}
R(-j)^{b_{1j}}\stackrel{}{\rightarrow} R
\rightarrow R/I \rightarrow 0.
\]
The 
{\it Castelnuovo-Mumford regularity\/} or simply the {\it regularity}
of $M$ is defined as $${\rm reg}(M)=\max\{j-i\vert\,
b_{ij}\neq 0\}.$$ 

An excellent reference for the regularity is the book of Eisenbud
\cite{eisenbud-syzygies}. There are methods to compute the regularity
of $R/I$ avoiding 
the construction of a minimal graded free resolution, see
\cite{bermejo-gimenez} and \cite[p.~614]{singular}. These methods
work for any homogeneous ideal over an arbitrary field.  
% and have been
%implemented in {\it Singular} \cite[p.~614]{singular}. Throughout this
%paper we use {\em Macaulay\/}$2$ \cite{mac2} to compute the 
%invariants of $R/I(\mathcal{C})$.

We are interested in finding good bounds for the 
regularity. Of particular interest is to be able to express ${\rm
  reg}(R/I(\mathcal{C}))$ in 
terms of the combinatorics of $\mathcal{C}$, at least for
some special families of clutters. Several authors have studied the
regularity of edge ideals of graphs 
and clutters
\cite{corso-nagel-tams,kia-faridi-traves,Ha-VanTuyl,
katzman1,kiani-moradi,kummini,mmcrty,nevo,terai,vantuyl,whieldon}.  The main results are general bounds for the regularity
and combinatorial formulas for the regularity of special families of
clutters. The estimates for the regularity are in terms of matching
numbers and the number of cliques needed to cover the vertex
set. Covers will play a particularly important role since they form
the basis for a duality.  

\begin{definition}\label{alex-dual}
The {\it ideal of covers} of $I(\mathcal{C})$, denoted by
$I_c(\mathcal{C})$, is the ideal of $R$ generated by 
all the monomials $x_{i_1}\cdots x_{i_k}$ such that
$\{x_{i_1},\ldots, x_{i_k}\}$ is a 
vertex cover of
$\mathcal{C}$. The ideal $I_c(\mathcal{C})$ is also called the {\it Alexander
dual\/} of $I(\mathcal{C})$ and is also denoted by
$I(\mathcal{C})^\vee$. The clutter of minimal vertex covers of 
$\mathcal{C}$, denoted by $\mathcal{C}^\vee$, is called the {\it Alexander
dual clutter} or {\it blocker} of $\mathcal{C}$.
\end{definition}
To better understand the Alexander dual, let $e\in E(\C)$ and consider
the monomial prime ideal $(e)=(\{x_i | x_i \in e\})$. Then the duality
is given by: 
\begin{equation}\label{alexander-duality}
\begin{array}{ccccc}
I(\mathcal{C})&=&(x_{e_1}, x_{e_2},\ldots,x_{e_q} )&=&
{\mathfrak p}_1 \cap {\mathfrak p}_2 \cap \cdots \cap {\mathfrak p}_s \\  
& & \updownarrow & & \updownarrow \\ 
I_c(\mathcal{C})&=&(e_1)\cap(e_2)\cap \cdots \cap (e_q)&=&
(x_{{\mathfrak p}_1}, x_{{\mathfrak p}_2},\ldots, x_{{\mathfrak p}_s}
), \\  
\end{array}
\end{equation}
where $\mathfrak{p}_1,\ldots,\mathfrak{p}_s$ are the 
associated primes of $I(\mathcal{C})$ and
$x_{\mathfrak{p}_k}=\prod_{x_i\in\mathfrak{p}_k}x_i$ for $1\leq k\leq
s$. Notice the equality $I_c(\mathcal{C})=I(\mathcal{C}^\vee)$. Since
$(\mathcal{C}^\vee)^\vee=\mathcal{C}$, we have
$I_c(\mathcal{C}^\vee)=I(\mathcal{C})$. 
In many cases $I(\mathcal{C})$ reflects properties of
$I_c(\mathcal{C})$ and viceversa \cite{ER,HeHi1,cca}. The following 
result illustrates this interaction.

\begin{theorem}{\rm \cite{terai}}\label{terai-formula} 
 Let $\mathcal{C}$ be a clutter. If ${\rm ht}(I(\mathcal{C}))\geq 2$, then  
$$
{\rm reg}\, I(\mathcal{C})=1+{\rm reg}\, R/I(\mathcal{C})={\rm pd}\, R/I_c(\mathcal{C}),
$$
where $I_c(\mathcal{C})$ is the ideal of minimal vertex covers of $\mathcal{C}$.
\end{theorem}

If $|e|\geq 2$
for all $e\in E(\mathcal{C})$, then this formula says that the
regularity of $R/I(\mathcal{C})$ equals $1$ if and only if
$I_c(\mathcal{C})$ is a Cohen-Macaulay ideal of height $2$. 
This formula will be used to
show that 
regularity behaves well when working with edge ideals with disjoint
sets of variables (see Proposition~\ref{additivity-reg}). This formula also holds for
edge ideals of height one \cite[Proposition 8.1.10]{herzog-hibi-book}.

\begin{corollary} If ${\rm ht}(I(\mathcal{C}))=1$, 
then ${\rm reg}\, R/I(\mathcal{C})={\rm pd}\, R/I_c(\mathcal{C})-1$.
\end{corollary}

\begin{proof} We set $I=I(\mathcal{C})$. The formula clearly holds if
$I=(x_1\cdots x_r)$ is a principal ideal. Assume that $I$ is not
principal. Consider the primary decomposition of $I$
$$
I=(x_1)\cap\cdots\cap(x_r)\cap\mathfrak{p}_1\cap\cdots\cap\mathfrak{p}_m,
$$
where $L=\mathfrak{p}_1\cap\cdots\cap\mathfrak{p}_m$ is an edge ideal of height
at least $2$. Notice that $I=fL$, where $f=x_1\cdots x_r$. Then the
Alexander dual of $I$ is
$$
I^\vee=(x_1,\ldots,x_r,x_{{\mathfrak p}_1}, x_{{\mathfrak
p}_2},\ldots, x_{{\mathfrak p}_m})=(x_1,\ldots,x_r)+L^\vee.
$$
The multiplication map $L[-r]\stackrel{f}{\rightarrow} fL$ induces an
isomorphism of graded $R$-modules. Thus ${\rm reg}(L[-r])=r+{\rm
reg}(L)={\rm reg}(I)$. By the Auslander-Buchsbaum
formula, one has the equality ${\rm
pd}(R/I^\vee)=r+{\rm pd}(R/L^\vee)$. Therefore, using
Theorem~\ref{terai-formula}, we get
$$
{\rm reg}(R/I)={\rm
reg}(R/L)+r\stackrel{\small\ref{terai-formula}}=({\rm pd}(R/L^\vee)-1)+r={\rm 
pd}(R/I^\vee)-1.
$$  
Thus ${\rm reg}(R/I)={\rm pd}(R/I^\vee)-1$, as required. 
\end{proof}

Next we show some basic properties of regularity. 
The first such property is
that regularity
behaves well when working with the edge ideal of a graph with multiple
disjoint components or with isolated vertices, as can be seen by the
following proposition. 

\begin{proposition} \cite[Lemma~7]{woodroofe-matchings} \label{additivity-reg} 
Let $R_1=K[\mathbf{x}]$ and 
$R_2=K[\mathbf{y}]$ be two polynomial rings over a field $K$ and let 
$R=K[\mathbf{x},\mathbf{y}]$. If $I_1$ and $I_2$ are edge ideals of
$R_1$ and $R_2$ respectively, then  
$$
{\rm reg}\, R/(I_1R+I_2R)={\rm reg}(R_1/I_1) +{\rm reg}(R_2/I_2).
$$
\end{proposition}

\begin{proof} By abuse of notation, we will write $I_i$ in place of
  $I_iR$ for $i=1,2$ when it is clear from context that we are using
  the generators of $I_i$ but extending to an
  ideal of the larger ring. Let $\mathbf{x}=\{x_1,\ldots,x_n\}$ and
$\mathbf{y}=\{y_1,\ldots,y_m\}$ be two disjoint sets of variables. 
Notice that $(I_1+I_2)^\vee=I_1^\vee I_2^\vee=I_1^\vee\cap I_2^\vee$
where $I_i^\vee$ is the Alexander dual of $I_i$ (see
Definition~\ref{alex-dual}). 
Hence by Theorem~\ref{terai-formula} and using the Auslander-Buchsbaum
formula, we get
\begin{eqnarray*}
{\rm reg}(R/(I_1+I_2))&=&n+m-{\rm
depth}(R/(I_1^\vee\cap I_2^\vee))-1,\\
{\rm reg}(R_1/I_1) +{\rm reg}(R_2/I_2)&=&n-{\rm depth}(R_1/I_1^\vee)-1+m-{\rm
depth}(R_2/I_2^\vee)-1
\end{eqnarray*}
Therefore we need only show the equality
$$
{\rm depth}(R/(I_1^\vee\cap I_2^\vee))={\rm depth}(R_1/I_1^\vee)+
{\rm depth}(R_2/I_2^\vee)+1.
$$
Since ${\rm depth}(R/(I_1^\vee+I_2^\vee))={\rm depth}(R_1/I_1^\vee)+
{\rm depth}(R_2/I_2^\vee)$, the
proof reduces to showing the equality
\begin{equation}\label{sept23-10-1}
{\rm depth}(R/(I_1^\vee\cap I_2^\vee))={\rm
depth}(R/(I_1^\vee+I_2^\vee))+1.
\end{equation}
We may assume that ${\rm depth}(R/I_1^\vee)\geq {\rm depth}(R/I_2^\vee)$. 
There is an exact sequence of graded $R$-modules: 
\begin{equation}\label{sept23-10}
0\longrightarrow R/(I_1^\vee\cap I_2^\vee) 
\stackrel{\varphi}{\longrightarrow}
R/I_1^\vee\oplus R/I_2^\vee\stackrel{\phi}{\longrightarrow}
R/(I_1^\vee+I_2^\vee)\longrightarrow 0,
\end{equation}
where $\varphi(\overline{r})=(\overline{r},-\overline{r})$ and 
$\phi(\overline{r}_1,\overline{r}_2)=\overline{r_1+r_2}$. From the
inequality
\begin{eqnarray*}
{\rm depth}(R/I_1^\vee\oplus R/I_2^\vee)&=&\max\{{\rm
depth}(R/I_i^\vee)\}_{i=1}^{2}={\rm
depth}(R/I_1^\vee)={\rm depth}(R_1/I_1^\vee)+m\\
&>&{\rm depth}(R_1/I_1^\vee)+{\rm
depth}(R_2/I_2^\vee)=
{\rm depth}(R/(I_1^\vee+I_2^\vee))
\end{eqnarray*}
and applying the depth lemma (see \cite[Proposition 1.2.9]{BHer} for
example) to Eq.~(\ref{sept23-10}), we obtain 
Eq.~(\ref{sept23-10-1}).
\end{proof}

Another useful property of regularity is that one can delete isolated
vertices of a graph without changing the regularity of the edge ideal. 
The following lemma 
shows that this can be done without significant changes to the
projective dimension as well. 
\begin{lemma}\label{addvariables}
Let $R=K[x_1, \ldots ,x_n]$ and $I$ be an ideal of $R$. If $I \subset
(x_1, \ldots , x_{n-1})$, and $R'=R/(x_n) \cong K[x_1, \ldots ,
x_{n-1}]$, then ${\rm reg}(R/I)={\rm reg}(R'/I)$ and ${\rm
pd}_R(R/I)={\rm pd}_{R'}(R'/I)$. Similarly, if $x_n \in I$ and
$I'=I/(x_n)$, then  ${\rm reg}(R/I)={\rm reg}(R'/I')$ and ${\rm
pd}_R(R/I)={\rm pd}_{R'}(R'/I')+1$.      
\end{lemma}

\begin{proof} The first result for projective dimension follows from
the Auslander-Buchsbaum formula since ${\rm depth}(R/I)={\rm
depth}(R'/I)+1$ and ${\rm depth}(R)={\rm depth}(R')+1$. Since ${\rm
depth}(R/I)={\rm depth}(R'/I')$ the second result for projective
dimension holds as well. The results for regularity follow from
Proposition~\ref{additivity-reg} by noting that the regularity of a
polynomial ring $K[x]$ is $0$, as is the regularity of the field $K=K[x]/(x)$.      
\end{proof}

While adding variables to the ring will preserve the regularity,
other changes to the base ring, such as changing the characteristic,
will affect this invariant. The following   
example shows that, even for graphs, a purely combinatorial description
of the regularity might not be possible. Results regarding the role the
characteristic of the field plays in the resolution of the ideal
appear in \cite{kia-manoj, katzman1}. 

\begin{example}\label{Gversusg-v}\rm Consider the edge ideal $I\subset
R=K[x_1,\ldots,x_{10}]$ generated by the monomials  
\[
\begin{array}{lllllllll}
x_1x_3,&x_1x_4,&x_1x_7,&x_1x_{10},&x_1x_{11},&x_2x_4,&x_2x_5,&
x_2x_8,&x_2x_{10},\\ x_2x_{11},& x_3x_5,&x_3x_6,&x_3x_8,&x_3x_{11},&x_4x_6,
&x_4x_9,&x_4x_{11},&x_5x_7,\\ x_5x_9,&x_5x_{11},&
x_6x_8,&x_6x_9,&x_7x_9,&x_7x_{10},&x_8x_{10}.& & 
\end{array}
\]

\noindent Using {\em Macaulay\/}$2$ \cite{mac2} we get that 
${\rm reg}(R/I)=3$ if ${\rm char}(K)=2$, and ${\rm reg}(R/I)=2$ if
${\rm char}(K)=3$.
\end{example}

As mentioned in Theorem~\ref{scm-clutters}(b), chordal graphs provide a key
example of a class of clutters whose edge ideals are sequentially
Cohen-Macaulay. Much work has been done toward finding hypergraph
generalizations of 
chordal graphs, typically by looking at cycles of edges or at tree
hypergraphs \cite{emtander,Ha-VanTuyl,Voloshin}.  The papers
\cite{Ha-VanTuyl,Voloshin,woodroofe-chordal} contribute to the algebraic
approach that is largely motivated by finding hypergraph
generalizations that have edge ideals with linear resolutions. 

It is
useful to consider the homogeneous components of the ideals when using
linear resolutions.
Let $(I_d)$ denote the ideal generated by all degree $d$ elements of
a homogeneous ideal $I$. Then $I$ is called {\it componentwise linear} if
$(I_d)$ has a linear
resolution for all $d$. If $I$ is the edge ideal of a clutter,
we write $I_{[d]}$ for the ideal generated by all the
squarefree monomials of degree $d$ in $I$.

\begin{theorem}\label{several-results}
Let $K$ be a
field and $\mathcal{C}$ be a clutter. Then
\begin{enumerate}
%\item[(a)] \cite{terai} ${\rm reg}\, I(\mathcal{C})=1+{\rm
%reg}\, R/I(\mathcal{C})={\rm pd}\, R/I_c(\mathcal{C})$. 
\item[(a)] \cite{ER} $R/I(\mathcal{C})$ is Cohen-Macaulay 
if and only if $I_c(\mathcal{C})$ has a linear resolution. 
\item[(b)] \cite{HeHi1} $R/I(\mathcal{C})$ is sequentially
Cohen-Macaulay if and only if $I_c(\mathcal{C})$ is componentwise linear.
\item[(c)] \cite{HeHi1} $I(\mathcal{C})$ is componentwise linear if and
only if $I(\mathcal{C})_{[d]}$ has a linear resolution for $d \geq 0$.
\item[(d)] \cite{Fro4} If $G$ is a graph, then $I(G)$ has a linear
resolution if and only if $G^c$ is chordal. 
\item[(e)] \cite{terai} If $G$ is a graph, then ${\rm reg}(R/I(G))=1$ if and 
only if $I_c(G)$ is Cohen-Macaulay.
\end{enumerate}
\end{theorem}

A graph whose complement is chordal is called {\it co-chordal}. A
consequence of this result and Theorem~\ref{terai-formula} is that an
edge ideal $I(\mathcal{C})$ has regularity $2$ if and only 
if $\Delta_\mathcal{C}$ is the independence complex of a co-chordal
graph. In this case the complex $\Delta_\mathcal{C}$ turns out to be a
quasi-forest in the sense of Zheng \cite{zheng}. In
\cite[Theorem~9.2.12]{herzog-hibi-book} it is shown that a complex
$\Delta$ is a quasi-forest if and only if $\Delta$ is the clique
complex of a chordal graph.

Information about the regularity of a clutter can also be found by
examining smaller, 
closely related clutters. Let $S$ be a set of vertices of a clutter $\mathcal{C}$.  The {\it induced
subclutter\/} on $S$, denoted by $\mathcal{C}[S]$, is the maximal
subclutter of $\mathcal{C}$ with vertex set $S$. Thus the vertex set
of $\mathcal{C}[S]$ is $S$ and the edges of $\mathcal{C}[S]$ are
exactly the edges of $\mathcal{C}$ contained in $S$. Notice that
$\mathcal{C}[S]$ may have isolated vertices, i.e., vertices that do
not belong to any edge of $\mathcal{C}[S]$. If $\mathcal{C}$ is a
discrete clutter, i.e., 
all the vertices of $\mathcal{C}$ are
isolated, we set $I(\mathcal{C})=0$ and $\alpha_0(\mathcal{C})=0$. A
clutter of the form $\mathcal{C}[S]$ for some $S\subset
V(\mathcal{C})$ is called an {\it induced subclutter} of
$\mathcal{C}$. 

\begin{proposition}\label{reg-monotonia} If $\mathcal{D}$ is an induced subclutter of
$\mathcal{C}$, then ${\rm reg}(R/I(\mathcal{D}))\leq{\rm
reg}(R/I(\mathcal{C}))$.
\end{proposition}

\begin{proof} There is $S\subset V(\mathcal{C})$ such that
$\mathcal{D}=\mathcal{C}[S]$. Let $\mathfrak{p}$ be the prime ideal of
$R$ generated by the variables in $S$. By duality, we have 
$$
I_c(\mathcal{C})=\bigcap_{e\in E(\mathcal{C})}\hspace{-2mm}(e)\ \Longrightarrow\ 
I_c(\mathcal{C})_\mathfrak{p}=\bigcap_{e\in
E(\mathcal{C})}\hspace{-2mm}(e)_\mathfrak{p}=
\bigcap_{e\in
E(\mathcal{D})}\hspace{-2mm}(e)_\mathfrak{p}=I_c(\mathcal{D})_\mathfrak{p}.
$$
Therefore, using Theorem~\ref{terai-formula} and Lemma~\ref{addvariables}, we get
\begin{eqnarray*}
{\rm reg}(R/I(\mathcal{C}))&=&{\rm pd}(R/I_c(\mathcal{C}))-1\\
&\geq &
{\rm pd}(R_\mathfrak{p}/I_c(\mathcal{C})_\mathfrak{p})-1=
{\rm pd}(R_\mathfrak{p}/I_c(\mathcal{D})_\mathfrak{p})-1\\
&=&
{\rm pd}(R'/I_c(\mathcal{D}))-1={\rm pd}(R/I_c(\mathcal{D}))-1={\rm
reg}(R/I(\mathcal{D})),
\end{eqnarray*}
where $R'$ is the polynomial ring $K[S]$. Thus, ${\rm
reg}(R/I(\mathcal{C}))\geq {\rm reg}(R/I(\mathcal{D}))$. 
\end{proof}

Several combinatorially defined invariants that bound the regularity
or other invariants of a clutter are given in terms of subsets of
the edge set of the clutter.  
An {\it induced matching\/} in a clutter $\mathcal{C}$ is a set of
pairwise disjoint edges $f_1,\ldots,f_r$ such that the only edges of 
$\mathcal{C}$ contained in $\cup_{i=1}^rf_i$ are $f_1,\ldots,f_r$. We
let ${\rm im}(\mathcal{C})$ be
the number of edges in the largest induced matching.

The next result was shown in \cite[Theorem~6.5]{Ha-VanTuyl} for the family
of uniform properly-connected hypergraphs. 

\begin{corollary}\label{lower-bound-reg} Let $\mathcal{C}$ be a
clutter and let $f_1,\ldots,f_r$ be an induced matching of 
$\mathcal{C}$ with $d_i=|f_i|$ for $i=1,\ldots,r$.  Then
\begin{enumerate}
\item[(a)] $(\sum_{i=1}^{r}d_i)-r\leq {\rm reg}(R/I(\mathcal{C}))$.
\item[(b)] \cite[Lemma~2.2]{katzman1} ${\rm im}(G)\leq {\rm
reg}(R/I(G))$ for any graph $G$. 
\end{enumerate}
\end{corollary}

\begin{proof} Let $\mathcal{D}=\mathcal{C}[\cup_{i=1}^rf_i]$. Notice
that $I(\mathcal{D})=(x_{f_1},\ldots,x_{f_r})$. Thus $I(\mathcal{D})$
is a complete intersection and the regularity of $R/I(\mathcal{D})$ is
the degree of its $h$-polynomial. The Hilbert series of $R/I(\mathcal{D})$
is given by 
$$
HS_{\mathcal{D}}(t)=\frac{\prod_{i=1}^r(1+t+\cdots+t^{d_i-1})}{(1-t)^{n-r}}.
$$
Thus, the degree of the $h$-polynomial equals
$(\sum_{i=1}^{r}d_i)-r$. Therefore, part (a) follows from 
Proposition~\ref{reg-monotonia}. Part (b) follows from part (a).
\end{proof}

\begin{corollary} If $\mathcal{C}$ is a clutter and $R/I_c(\mathcal{C})$
is Cohen-Macaulay, then ${\rm im}(\mathcal{C})=1$.
\end{corollary}

\begin{proof} Let $r$ be the induced matching number of $\mathcal{C}$
and let $d$ be the cardinality of any edge of $\mathcal{C}$. Using
Theorem~\ref{terai-formula}  and Corollary~\ref{lower-bound-reg}, we
obtain $d-1\geq r(d-1)$. Thus $r=1$, as required.
\end{proof}

The following example shows that the inequality obtained in
Corollary~\ref{lower-bound-reg}(b) can be strict. 

\begin{example} Let $G$ be the complement of a cycle
${C}_6=\{x_1,\ldots,x_6\}$ of length six.
The edge ideal of $G$ is  
$$
I(G)=(x_1x_3,x_1x_5,x_1x_4,x_2x_6,x_2x_4,x_2x_5,x_3x_5,x_3x_6,x_4x_6).
$$
Using {\em Macaulay\/}$2$ \cite{mac2}, we get ${\rm reg}(R/I(G))=2$ and ${\rm im}(G)=1$.
\end{example}

\begin{lemma}{\cite[Corollary~20.19]{Eisen}}\label{reg-lemma} If
$0\rightarrow N\rightarrow M\rightarrow L\rightarrow 0$ is a  
short exact sequence of graded finitely generated $R$-modules, then
\begin{itemize}
\item[(a)] ${\rm reg}(N)\leq\max({\rm reg}(M),{\rm reg}(L)+1)$.
\item[(b)] ${\rm reg}(M)\leq\max({\rm reg}(N),{\rm reg}(L))$.
\item[(c)] ${\rm reg}(L)\leq\max({\rm reg}(N)-1,{\rm reg}(M))$.
\end{itemize}
\end{lemma}

\begin{definition}
If $x$ is a vertex of a graph $G$, then
its {\it neighbor set},  denoted by $N_G(x)$, is
the set of vertices of $G$ adjacent to $x$.   
\end{definition}

The following theorem gives a precise sense in which passing to
induced subgraphs can be used to bound the regularity. Recall that a
discrete graph is one in which all the vertices are isolated.

\begin{theorem}\label{reg-of-families}
Let $\mathcal{F}$ be a family of graphs containing any discrete graph
and let $\beta\colon\mathcal{F}\rightarrow\mathbb{N}$ be a function
satisfying that $\beta(G)=0$ for any discrete graph $G$, and such 
that given $G\in\mathcal{F}$, with $E(G)\neq\emptyset$, there is $x\in V(G)$
such that the following two conditions hold{\rm :} 
\begin{enumerate}
\item[(i)] $G\setminus\{x\}$ and $G\setminus(\{x\}\cup N_G(x))$ are in
$\mathcal{F}$.
\item[(ii)] $\beta(G\setminus(\{x\}\cup N_G(x)))<\beta(G)$ and
$\beta(G\setminus\{x\})\leq \beta(G)$.
\end{enumerate} 
Then ${\rm reg}(R/I(G))\leq \beta(G)$ for any $G\in\mathcal{F}$.
\end{theorem}

\begin{proof} The proof is by induction on the number of vertices. Let $G$ be a graph in
$\mathcal{F}$. If $G$ is a discrete graph, then $I(G)=(0)$ 
and ${\rm reg}(R)=\beta(G)=0$. Assume that $G$ has at least one
edge.  There is a vertex $x\in V(G)$ such that the induced subgraphs 
$G_1=G\setminus\{x\}$ and $G_2=G\setminus(\{x\}\cup N_G(x))$ satisfy
(i) and (ii). There is an exact sequence of graded $R$-modules 
$$
0\longrightarrow R/(I(G)\colon x)[-1]\stackrel{x}{\longrightarrow}
R/I(G)\longrightarrow 
R/(x,I(G))\longrightarrow 0.
$$
Notice that $(I(G)\colon x)=(N_G(x),I(G_2))$ and
$(x,I(G))=(x,I(G_1))$. The graphs $G_1$ and $G_2$ have fewer vertices
than $G$. It follows directly from the definition of
regularity that ${\rm reg}(M[-1])=1+{\rm reg}(M)$ for any graded 
$R$-module $M$. Therefore applying the induction hypothesis to $G_1$
and $G_2$, and using conditions (i) and (ii) and Lemma~\ref{addvariables}, we get 
\begin{eqnarray*}
{\rm reg}(R/(I(G)\colon x)[-1])&=&{\rm reg}(R/(I(G)\colon x))+1 = {\rm reg}(R'/I(G_2))+1  \leq 
\beta(G_2)+1\leq \beta(G),\\ 
{\rm reg}(R/(x,I(G)))&\leq&
\beta(G_1)\leq \beta(G)
\end{eqnarray*}
where $R'$ is the ring in the variables $V(G_2)$.
Therefore from Lemma~\ref{reg-lemma}, we get that the
regularity of $R/I(G)$ is bounded by the maximum of the regularities
of $R/(I(G)\colon x)[-1]$ and $R/(x,I(G))$. Thus 
${\rm reg}(R/I(G))\leq \beta(G)$, as required. 
\end{proof}

As an example of how Theorem~\ref{reg-of-families} can be applied to
obtain combinatorial bounds for the regularity, we provide new proofs
for several previously known results. Let $G$ be a graph. We let
$\beta'(G)$ be the cardinality of any 
smallest maximal matching of $G$.  H$\rm \grave{a}$ and Van Tuyl
proved that the regularity of $R/I(G)$ is bounded from above by the
matching number of $G$ and  Woodroofe improved this result showing
that $\beta'(G)$ is an upper bound for the regularity. 

\begin{corollary}\label{chordal-reg} Let $G$ be a graph and let $R=K[V(G)]$. Then
\begin{enumerate}
\item[(a)] \cite[Corollary~6.9]{Ha-VanTuyl} ${\rm reg}(R/I(G))={\rm
im}(G)$ for any chordal graph $G$.   

\item[(b)] $($\cite[Theorem~6.7]{Ha-VanTuyl},
\cite{woodroofe-matchings}$)$ ${\rm reg}(R/I(G))\leq \beta'(G)$.

\item[(c)] \cite[Theorem~3.3]{vantuyl} ${\rm reg}(R/I(G))={\rm
im}(G)$ if $G$ is bipartite and $R/I(G)$ is sequentially Cohen-Macaulay. 
\end{enumerate}
\end{corollary}

\begin{proof} (a) Let $\mathcal{F}$ be the family of chordal graphs
and let $G$ be a chordal graph with $E(G)\neq\emptyset$. 
By Corollary~\ref{lower-bound-reg} and Theorem~\ref{reg-of-families}
it suffices to prove 
that there is $x\in V(G)$ such that
${\rm im}(G_1)\leq {\rm im}(G)$ and
${\rm im}(G_2)<{\rm im}(G)$, where $G_1$ and $G_2$ are the 
subgraphs $G\setminus\{x\}$ and $G\setminus(\{x\}\cup N_G(x))$,
respectively. The inequality ${\rm im}(G_1)\leq {\rm im}(G)$ is clear
because any induced matching of $G_1$ is an induced matching of $G$.
We now show the other inequality. By \cite[Theorem~8.3]{Toft}, there is $y\in V(G)$ such
that $G[N_G(y)\cup\{y\}]$ is a complete subgraph. Pick $x\in N_G(y)$
and set $f_0=\{x,y\}$. Consider an induced matching $f_1,\ldots,f_r$
of $G_2$ with $r={\rm im}(G_2)$. We claim that $f_0,f_1,\ldots,f_r$ is
an induced matching of $G$. Let $e$ be an edge of $G$ contained in
$\cup_{i=0}^rf_i$. We may assume that $e\cap f_0\neq\emptyset$ and
$e\cap f_i\neq\emptyset$ for some $i\geq 1$, otherwise $e=f_0$ or
$e=f_i$ for some $i\geq 1$. Then $e=\{y,z\}$ or $e=\{x,z\}$ for some
$z\in f_i$. If $e=\{y,z\}$, then $z\in N_G(y)$ and $x\in N_G(y)$.
Hence $\{z,x\}\in E(G)$ and $z\in N_G(x)$, a contradiction because the
vertex set of $G_2$ is disjoint from $N_G(x)\cup\{x\}$. If
$e=\{x,z\}$, then $z\in N_G(x)$, a contradiction. This completes the
proof of the claim. Hence ${\rm im}(G_2)<{\rm im}(G)$.

(b) Let $\mathcal{F}$ be the family of all graphs and let $G$ be a
graph with $E(G)\neq\emptyset$. By Theorem~\ref{reg-of-families} it
suffices to prove that there is $x\in V(G)$ such that
$\beta'(G_1)\leq \beta'(G)$ and
$\beta'(G_2)<\beta'(G)$, where $G_1$ and $G_2$ are the 
subgraphs $G\setminus\{x\}$ and $G\setminus(\{x\}\cup N_G(x))$,
respectively. 

Let $f_1,\ldots,f_r$ be a maximal matching of $G$ with $r=\beta'(G)$ and let
$x,y$ be the vertices of $f_1$. Clearly $f_2,\ldots,f_r$ is a 
matching of $G_1$. Thus we can extend it to a maximal 
matching  $f_2,\ldots,f_r,h_1,\ldots,h_s$ of $G_1$. Notice that 
$s\leq 1$. Indeed if $s\geq 2$, then $h_i\cap f_1=\emptyset$ for 
some $i\in\{1,2\}$ (otherwise $y\in h_1\cap h_2$, which is
impossible). Hence $f_1,\ldots,f_r,h_i$ is a matching of 
$G$, a contradiction because $f_1,\ldots,f_r$ is maximal. Therefore
$\beta'(G_1)\leq r-1+s\leq\beta'(G)$.  

The set $f_2,\ldots,f_r$ contains a matching of $G_2$, namely those
edges $f_i$ that do not degenerate. Reorder the edges so that $f_2,
\ldots , f_m$ are the edges that do not degenerate. Then this set can
be extended to a maximal matching $f_2, \ldots , f_m, f_{m+1}',
\ldots , f_k'$ of $G_2$. Now consider $f_{m+1}'$. Since $f_1, \dots ,
f_r$ is a maximal matching of $G$, $f_{m+1}'$ has a nontrivial
intersection with $f_i$ for some $i$. Note that $i\not= 1$ since
$f_{m+1}'$ is an edge of $G_2$, and $i \geq m+1$ since $f_2, \ldots ,
f_m$ and $f_{m+1}'$ are all part of a matching of $G_2$. Reorder so
that $i=m+1$. Repeat the process with $f_{m+2}'$. As before,
$f_{m+2}'$ has a nontrivial intersection with $f_i$ for some $i\geq
m+1$. If $i=m+1$, then since $f_{m+1}'$ and $f_{m+2}'$ are disjoint,
each must share a different vertex with $f_{m+1}$. But $f_{m+1}$
degenerated 
when passing to $G_2$, and $G_2$ is induced on the remaining
vertices, so this is a contradiction. Thus we may reorder so that
$f_{m+2}'$ 
nontrivially intersects $f_{m+2}$. Repeating 
this process, we see that $f_j'$ nontrivially 
intersects $f_j$ for all $m+1 \leq j \leq k$. 
Thus $k \leq r$. Now $\beta'(G_2) \leq k-1 \leq r-1 < \beta'(G)$.

(c) Let $\mathcal{F}$ be the family of all bipartite graphs $G$ 
such that $R/I(G)$ is sequentially Cohen-Macaulay, and let
$\beta\colon\mathcal{F}\rightarrow\mathbb{N}$ be the function
$\beta(G)={\rm im}(G)$. Let $G$ be a graph
in $\mathcal{F}$ with $E(G)\neq\emptyset$. By
Corollary~\ref{lower-bound-reg} and Theorem~\ref{reg-of-families}
it suffices to observe that, according to \cite[Corollary
2.10]{bipartite-scm}, there are adjacent vertices $x$ and $y$ with
${\rm deg}(y) = 1$ such that the bipartite graphs
$G\setminus(\{x\}\cup N_G(x))$ and $G\setminus(\{y\}\cup N_G(y))$ are
sequentially Cohen-Macaulay. Thus conditions (i) and (ii) of
Theorem~\ref{reg-of-families} are satisfied. 
\end{proof}

Corollary~\ref{chordal-reg} shows that the regularity of $R/I(G)$
equals ${\rm im}(G)$ for any forest $G$, which was first proved by
Zheng \cite{zheng}. 
If $G$ is an unmixed graph, Kummini
\cite{kummini} showed that ${\rm reg}(R/I(G))$ equals the induced
matching number of $G$. If $G$ is claw-free and its complement
has no induced $4$-cycles, then ${\rm reg}(R/I(G))\leq 2$ with equality
if its complement is not chordal \cite{nevo} (note that in this case 
${\rm reg}(R/I(G))={\rm im}(G) + 1$). Formulas for the regularity 
of ideals of mixed products are given in \cite{ionescu-rinaldo}. 
The regularity and depth of lex segment edge ideals are computed in
\cite{terai-reg-ex}. The regularity and other algebraic properties of
the edge ideal $I(G)$ associated to a
Ferrers graph $G$ are studied in detail in \cite{corso-nagel-tams}.
If $R/I(G)$ is Cohen-Macaulay, 
then a formula for ${\rm reg}(R/I(G))$ follows 
from \cite[Corollary~4.2]{renteln}.  

The following result about regularity was shown  by
Kalai and Meshulam for square-free monomial ideals and by Herzog for
arbitrary monomial ideals. Similar inequalities hold for the
projective dimension.

\begin{proposition} \cite{kalai-meshulam,herzog-reg} \label{kalai-herzog}
Let $I_1$ and $I_2$ be monomial ideals of $R$. Then 
\begin{enumerate}
\item[(a)] ${\rm reg}\, R/(I_1+I_2)\leq {\rm reg}(R/I_1) +{\rm
reg}(R/I_2)$,
\item[(b)] ${\rm reg}\, R/(I_1\cap I_2)\leq {\rm reg}(R/I_1) +{\rm
reg}(R/I_2)+1$.
\end{enumerate}
\end{proposition}

\begin{corollary}\label{reg-union} If
  $\mathcal{C}_1,\ldots,\mathcal{C}_s$ are clutters on 
the vertex set $X$, then 
$${\rm reg}(R/I(\cup_{i=1}^s\mathcal{C}_i))\leq {\rm
  reg}(R/I(\mathcal{C}_1))+\cdots+{\rm
  reg}(R/I(\mathcal{C}_s)).  
$$
\end{corollary}

\begin{proof} The set of edges of
  $\mathcal{C}=\cup_{i=1}^s\mathcal{C}_i$ equals 
$\cup_{i=1}^sE(\mathcal{C}_i)$. By Proposition~\ref{kalai-herzog}, it
  suffices to notice 
the equality $I(\cup_{i=1}^s\mathcal{C}_i)=\sum_{i=1}^{s}I(\mathcal{C}_i)$.
\end{proof}

A clutter $\mathcal{C}$ is called {\it co-CM} if $I_c(\mathcal{C})$
is Cohen-Macaulay. A co-CM clutter is uniform because Cohen-Macaulay
clutters are unmixed. 

\begin{corollary} If $\mathcal{C}_1,\ldots,\mathcal{C}_s$ are co-CM clutters on
the vertex set $X$, then 
$${\rm reg}(R/I(\cup_{i=1}^s\mathcal{C}_i))\leq
(d_1-1)+\cdots+(d_s-1),
$$
where $d_i$ is the number of elements in any edge of $\mathcal{C}_i$.
\end{corollary}

\begin{proof} By Theorem~\ref{terai-formula}, we get 
that ${\rm reg}\, R/I(\mathcal{C}_i)=d_i-1$ for all $i$. Thus the
result follows from Corollary~\ref{reg-union}. 
\end{proof}

This result is especially useful for 
graphs. A graph $G$ is {\it weakly chordal} if every induced cycle
in both $G$ and $G^c$ has length at most $4$. It was pointed out in
\cite{woodroofe-matchings} that 
a weakly chordal graph $G$ can
be covered by ${\rm im}(G)$ co-CM graphs (this fact was shown in
\cite{co-chordal}). Thus we have:

\begin{theorem} \cite{woodroofe-matchings}
If $G$ is a weakly chordal graph, then ${\rm reg}(R/I(G))={\rm
im}(G)$.
\end{theorem}

There are bounds for the regularity of $R/I$ in terms of some other
algebraic invariants of $R/I$. Recall that the $a$-invariant of
$R/I$, denoted by $a(R/I)$, is the degree (as a
rational function) of the Hilbert series of $R/I$. Also recall that the
independence complex of 
$I(\mathcal{C})$, denoted by $\Delta_\mathcal{C}$, is the simplicial
complex whose faces are the independent vertex sets of $\mathcal{C}$.
The {\it arithmetic
degree} of $I=I(\mathcal{C})$, denoted by ${\rm arith}$-${\rm
  deg}(I)$, is the number of 
{\it facets\/} (maximal faces with respect to inclusion) of
$\Delta_\mathcal{C}$. The {\it arithmetical rank} of $I$, denoted by ${\rm
ara}(I)$, is the least number of elements of $R$ which generate the
ideal $I$ up to radical. 

\begin{theorem}{\rm\cite[Corollary~B.4.1]{Vas1}} $a(R/I)\leq{\rm
reg}(R/I)-{\rm depth}(R/I)$, with equality if $R/I$ is Cohen-Macaulay.
\end{theorem}

\begin{theorem}{\rm(\cite{lyubeznik-ara}, \cite[Proposition~3]{Lyu3})} ${\rm reg}(R/I^\vee)={\rm
pd}(R/I)-1\leq {\rm ara}(I)-1$.
\end{theorem}

The equality ${\rm reg}(R/I^\vee)={\rm
pd}(R/I)-1$ was pointed out earlier in
Theorem~\ref{terai-formula}.  There are many instances where the
equality ${\rm pd}(R/I)={\rm
ara}(I)$ holds, see \cite{barile,terai-reg-ex,terai-yoshida-small} and
the references there. For example,
for paths, one has ${\rm pd}(R/I)={\rm ara}(I)$ \cite{barile}. 
Barile \cite{barile} has conjectured that the
equality holds for edge ideals of forests. We also have that $n-\min_i\{{\rm
depth}(R/I^{(i)})\}$ is an upper bound for ${\rm ara}(I)$, see
\cite{Lyu3}. This upper bound tends to be very loose. 
If $I$ is
the edge ideal of a tree, then $I$ is normally torsion free (see
Section~\ref{stability} together with Theorems~\ref{noclu1} and \ref{bipartite}).
Then $\min_i\{{\rm depth}(R/I^{(i)})\}=1$ by \cite[Lemma~2.6]{morey}.  But when
$I$ is the edge ideal of a path with $8$ vertices, then the actual value
of ${\rm ara}(I)$ is $5$. 

\begin{theorem}{\rm\cite[Theorem~3.1]{terai}} If ${\rm ht}(I)\geq 2$,
then ${\rm reg}(I)\leq {\rm arith}$-${\rm deg}(I)$. 
\end{theorem}

The next open problem is known as the Eisenbud-Goto regularity
conjecture \cite{eisenbud-goto}.

\begin{conjecture} If $\mathfrak{p}\subset
(x_1,\ldots,x_n)^2$ is a prime graded ideal, then 
$${\rm
reg}(R/\mathfrak{p})\leq {\rm
deg}(R/\mathfrak{p})-{\rm codim}(R/\mathfrak{p}).
$$
\end{conjecture}

A pure $d$-dimensional complex $\Delta$ is
called {\it connected in codimension\/} $1$ if each pair of 
facets $F,G$ can be connected by a sequence of
facets $F=F_0,F_1,\ldots,F_s=G$, such that $\dim(F_{i-1}\cap F_i)=d-1$
for $1\leq i\leq s$. According to \cite[Proposition~11.7]{bjorner-topological}, every
Cohen-Macaulay complex is connected in codimension $1$.

The following gives a partial answer to the monomial version of the
Eisenbud-Goto regularity conjecture.

\begin{theorem}{\cite{terai}} Let $I=I(\mathcal{C})$ be an edge 
ideal. If $\Delta_\mathcal{C}$ is connected in codimension $1$, then  
$${\rm
reg}(R/I)\leq {\rm
deg}(R/I)-{\rm codim}(R/I).
$$
\end{theorem}

The dual notion to the independence complex of $I(\mathcal{C})$ is to
start with a complex $\Delta$ and associate to it an ideal whose
independence complex is $\Delta$.

\begin{definition} Given a simplicial complex $\Delta$ with vertex set
  $X=\{x_1,\ldots ,x_n\}$,  
the {\em Stanley-Reisner ideal\/} of $\Delta$ is defined as 
$$I_{\Delta}=\left(\{x_{i_1}\cdots x_{i_r}\vert\ i_1<\cdots < i_r 
{\rm,}\ \{x_{i_1},\ldots, x_{i_r}\}\notin \Delta\}\right),
$$ 
and its {\em Stanley-Reisner ring\/} $K[\Delta]$ 
is defined as the quotient ring $R/I_\Delta$.
\end{definition}

A simple proof the next result is given in \cite{Fro4}.

\begin{theorem} \cite{dsmith} \label{aug21-01}
Let $\mathcal{C}$ be a clutter and let $\Delta=\Delta_\mathcal{C}$ be
its independence complex. Then 
$$
{\rm depth}\, R/I(\mathcal{C})=1+\max\{i\, \vert\, K[\Delta^i]
\mbox{ is Cohen-Macaulay}\},
$$
where $\Delta^i=\{F\in\Delta\, \vert\, 
\dim(F)\leq i\}$ is the $i$-skeleton of $\Delta$ and $-1\leq
i\leq\dim(\Delta)$.
\end{theorem}

A variation on the concept of the $i$-skeleton will facilitate an
extension of the result above to the sequentially Cohen-Macaulay case.

\begin{definition} Let $\Delta$ be a simplicial complex. The {\it
pure $i$-skeleton\/} of 
$\Delta$ is defined as:
\begin{eqnarray*}
\Delta^{[i]}&=&\langle\{F\in\Delta\vert\, \dim(F)=i\}\rangle;\ \ -1\leq
i\leq\dim(\Delta),
\end{eqnarray*}
where $\langle{\mathcal F}\rangle$ denotes the subcomplex generated by 
$\mathcal F$.  
\end{definition}

Note that $\Delta^{[i]}$ is always pure of dimension $i$. We say that
a simplicial 
complex $\Delta$ 
is {\it sequentially Cohen-Macaulay} if its Stanley-Reisner ring has
this property. The following results link the sequentially
Cohen-Macaulay property to the Cohen-Macaulay property and to the
regularity  and projective dimension. The first is an interesting
result of Duval.  

\begin{theorem}{\rm \cite[Theorem~3.3]{duval}}\label{duval-theorem} A
simplicial complex $\Delta$ is  
sequentially Cohen-Macaulay if and only if the pure $i$-skeleton
$\Delta^{[i]}$  
is Cohen-Macaulay for  $-1\leq i\leq \dim(\Delta)$. 
\end{theorem}

\begin{corollary} $R/I(\mathcal{C})$ is Cohen-Macaulay if and only if $R/I(\mathcal{C})$ is sequentially
Cohen-Macaulay and $\mathcal{C}$ is unmixed.  
\end{corollary}

\begin{lemma}\label{pure-skel-vs-skel} Let $\mathcal{C}$ be a clutter and let
$\Delta=\Delta_\mathcal{C}$ be its independence complex. If
$\beta_0'(\mathcal{C})$ is the cardinality of a 
smallest maximal independent set of $\mathcal{C}$, then
$\Delta^{[i]}=\Delta^i$ for 
$i\leq\beta_0'(\mathcal{C})-1$. 
\end{lemma}

\begin{proof} First we prove the inclusion 
$\Delta^{[i]}\subset\Delta^i$. Let $F$ be a face of $\Delta^{[i]}$.
Then $F$ is contained in a face of $\Delta$ of dimension $i$, and so
$F$ is in $\Delta^i$. Conversely, let $F$ be a face of $\Delta^i$.
Then
$$
\dim(F)\leq i\leq\beta_0'(\mathcal{C})-1\ \Longrightarrow\ |F|\leq
i+1\leq\beta_0'(\mathcal{C}) .
$$
Since $\beta_0'(\mathcal{C})$ is the cardinality of any 
smallest maximal independent set of $\mathcal{C}$, we can extend $F$
to an independent set of $\mathcal{C}$ with $i+1$ vertices. Thus $F$
is in $\Delta^{[i]}$. 
\end{proof}

While $\beta_0'$ regulates the equality of the $i$-skeleton and the
pure $i$-skeleton of the independence complex, its complement provides
a lower bound for the regularity of the ideal of covers.  

\begin{theorem}\label{reg-vila-morey-scm} Let $\mathcal{C}$ be a
clutter, 
let $I_c(\mathcal{C})$ be its ideal of vertex covers, and let
$\alpha_0'(\mathcal{C})$ be the cardinality of a largest minimal vertex
cover of $\mathcal{C}$. Then 
$$ 
{\rm reg}\, R/I_c(\mathcal{C})\geq \alpha_0'(\mathcal{C})-1,
$$
with equality if $R/I(\mathcal{C})$ is sequentially Cohen-Macaulay. 
\end{theorem}

\begin{proof} We set 
$\beta_0'(\mathcal{C})=n-\alpha_0'(\mathcal{C})$. Using Theorem~\ref{terai-formula} and the
Auslander-Buchsbaum formula (see Equation (\ref{auslander-buchsbaum-formula})), the proof
reduces to showing: ${\rm depth}\, R/I(\mathcal{C})\leq \beta_0'(\mathcal{C})$, with equality 
if $R/I(\mathcal{C})$ is sequentially Cohen-Macaulay. 

First we show that ${\rm
depth}\, R/I(\mathcal{C})\leq\beta_0'(\mathcal{C})$.  Assume
$\Delta^i$ is Cohen-Macaulay for some $-1\leq
i\leq\dim(\Delta)$, where $\Delta$ is the independence
complex of $\mathcal{C}$. According to 
Theorem~\ref{aug21-01}, it suffices to prove that 
$1+i\leq\beta_0'(\mathcal{C})$. Notice that $\beta_0'(\mathcal{C})$ is the cardinality of any 
smallest maximal independent set of $\mathcal{C}$. Thus, we can pick 
a maximal independent set $F$ of $\mathcal{C}$ with $\beta_0'(\mathcal{C})$ 
vertices. Since $\Delta^i$ is Cohen-Macaulay, the complex $\Delta^i$ 
is pure, that is, all maximal faces of $\Delta$ have 
dimension $i$. If $1+i>\beta_0'(\mathcal{C})$, then $F$ is a maximal
face of $\Delta^i$ of dimension $\beta_0'(\mathcal{C})-1$, a
contradiction to the purity of $\Delta^i$. 

Assume that $R/I(\mathcal{C})$ is sequentially Cohen-Macaulay. By
Lemma~\ref{pure-skel-vs-skel} $\Delta^{[i]}=\Delta^i$ for
$i\leq \beta_0'(\mathcal{C})-1$. Then by
Theorem~\ref{duval-theorem}, the ring $K[\Delta^i]$ is Cohen-Macaulay
for $i\leq \beta_0'(\mathcal{C})-1$. Therefore, applying
Theorem~\ref{aug21-01}, we get that the depth of
$R/I(\mathcal{C})$ is at least $\beta_0'(\mathcal{C})$. Consequently,
in this case one has the equality ${\rm depth}\,
R/I(\mathcal{C})=\beta_0'(\mathcal{C})$.  
\end{proof}

The inequality in Theorem~\ref{reg-vila-morey-scm} also follows 
directly from the definition of regularity because ${\rm
reg}(I_c(\mathcal{C}))$ is an upper bound for the largest degree of a
minimal generator of $I_c(\mathcal{C})$.   

\begin{remark}\label{remark-alpha} $\alpha_0'(\mathcal{C})$ is $\max\{|e|\colon 
e\in E(\mathcal{C}^\vee)\}$ and $\alpha_0'(\mathcal{C}^\vee)$ is $\max\{|e|\colon 
e\in E(\mathcal{C})\}$. This follows by Alexander duality, see
Eq.~(\ref{alexander-duality}).  
\end{remark}

\begin{corollary}\label{vila-morey-pd} If $I(\mathcal{C})$ is an edge ideal, 
then ${\rm pd}_R(R/I(\mathcal{C}))\geq\alpha_0'(\mathcal{C})$, with
equality if $R/I(\mathcal{C})$ is sequentially Cohen-Macaulay.
\end{corollary}

\begin{proof} It follows from the proof of
Theorem~\ref{reg-vila-morey-scm}.
\end{proof}

%\newpage

There are many interesting classes of sequentially Cohen-Macaulay
clutters where this formula for the projective dimension applies 
(see Theorem~\ref{scm-clutters}). The projective dimension of edge
ideals of forests was studied in \cite{dochtermann,Ha-VanTuyl}, where
some recursive formulas are presented. Explicit formulas for the
projective dimension for some path ideals of directed rooted trees
can be found in \cite[Theorem~1.2]{He-VanTuyl}. Path ideals of
directed graphs were introduced by Conca and De Negri
\cite{ConcaDeNegri}. Fix an integer $t\geq 2$, and suppose that
$\mathcal{D}$ is a directed graph, i.e., each edge
has been assigned a direction.  A sequence of $t$ vertices
$x_{i_1},\ldots,x_{i_t}$ is said to be a {\it path} of length
$t$ if there are $t-1$ distinct edges $e_1,\ldots,e_{t-1}$ such that
$e_j = (x_{i_j},x_{i_{j+1}})$ is a directed
edge from $x_{i_j}$ to $x_{i_{j+1}}$.  The {\it path ideal}
of $\mathcal{D}$ of length $t$, denoted by $I_t(\mathcal{D})$, is the
ideal generated by
all monomials $x_{i_1}\cdots x_{i_t}$ such that $x_{i_1},\ldots,
x_{i_t}$ is a path of length $t$ in $\mathcal{D}$.
Note that when $t=2$, then $I_2(\mathcal{D})$ is simply the
edge ideal of $\mathcal{D}$.

\begin{example}\label{exvavi}\rm Let $K$ be any field and let
$G$ be the following chordal graph

\vspace{1.5cm}
%\begin{center}
$
\setlength{\unitlength}{.020cm}
\thicklines
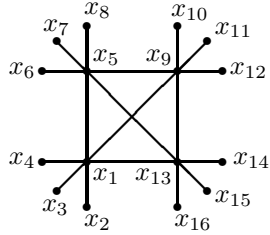
\begin{picture}(-15,0)(-300,0)

\put(-30,-30){\circle*{6.1}}
\put(30,30){\circle*{6.1}}
\put(30,-30){\circle*{6.1}}
\put(-30,30){\circle*{6.1}}

\put(-60,-30){\circle*{6.1}}
\put(-30,-60){\circle*{6.1}}
\put(-50,-50){\circle*{6.1}}

\put(60,30){\circle*{6.1}}
\put(30,60){\circle*{6.1}}
\put(50,50){\circle*{6.1}}

\put(-60,30){\circle*{6.1}}
\put(-30,60){\circle*{6.1}}
\put(-50,50){\circle*{6.1}}

\put(60,-30){\circle*{6.1}}
\put(30,-60){\circle*{6.1}}
\put(50,-50){\circle*{6.1}}

\put(-60,-30){\line(1,0){120}}
\put(-30,-60){\line(0,1){120}}
\put(-50,-50){\line(1,1){100}}
\put(60,30){\line(-1,0){120}}
\put(30,60){\line(0,-1){120}}
\put(-50,50){\line(1,-1){100}}

\put(-26,-42){\small$x_1$}
\put(9,37){\small$x_9$}
\put(-26,37){\small$x_5$}
\put(2,-42){\small$x_{13}$}

\put(-82,-32){\small$x_4$}
\put(-60,-63){\small$x_3$}
\put(-32,-73){\small$x_2$}

\put(65,28){\small$x_{12}$}
\put(55,55){\small$x_{11}$}
\put(26,67){\small$x_{10}$}

\put(67,-32){\small$x_{14}$}
\put(55,-57){\small$x_{15}$}
\put(28,-73){\small$x_{16}$}

\put(-82,28){\small$x_6$}
\put(-59,57){\small$x_7$}
\put(-32,68){\small$x_8$}

\end{picture}
$
%\end{center}
\vspace{1.7cm}

\noindent Then, by Theorem~\ref{scm-clutters} and
Corollary~\ref{reg-vila-morey-scm}, 
we get ${\rm pd}_R(R/I(G))=6$ and ${\rm depth}\, R/I(G)=10$.  
\end{example}

\begin{corollary}\label{reg-linear-quotients} Let $\mathcal{C}$ be a
clutter. If $I(\mathcal{C})$ has linear quotients, then 
$$
{\rm reg}\, R/I(\mathcal{C})=\max\{|e|\colon\, e\in E(\mathcal{C})\}-1.
$$ 
\end{corollary} 

\begin{proof} The ideal of covers $I_c(\mathcal{C})$ is sequentially 
Cohen-Macaulay by Theorem~\ref{scm-clutters}(d). Hence, using
Theorem~\ref{reg-vila-morey-scm}, we get ${\rm
reg}(R/I(\mathcal{C}))=\alpha_0'(\mathcal{C}^\vee)-1$. To complete the
proof notice that $\alpha_0'(\mathcal{C}^\vee)=\max\{|e|\colon\, e\in
E(\mathcal{C})\}$ (see Remark~\ref{remark-alpha}).
\end{proof}

The converse of Theorem~\ref{vila-morey-pd} is not true.

\begin{example} Let $C_6$ be a cycle of length $6$. Then $R/I(C_6)$ is not
sequentially Cohen-Macaulay by Proposition~\ref{only-scm-cycles}. Using
{\em Macaulay\/}$2$, we get ${\rm pd}(R/I(C_6))=\alpha_0'(C_6)=4$.
\end{example}

When $R/I$ is not known to be Cohen-Macaulay, it can prove useful to
have effective bounds on the depth of $R/I$. 

\begin{theorem} Let $G$ be a bipartite graph without 
isolated vertices. If $G$ has $n$ vertices, then
$$
{\rm depth}\, R/I(G)\leq
\left\lfloor\frac{n}{2}\right\rfloor.
$$
\end{theorem}

\begin{proof}
Let $(V_1,V_2)$ be a bipartition of $G$ with $|V_1|\leq |V_2|$.
Note $2|V_1|\leq n$ because  
$|V_1|+|V_2|=n$. Since $V_1$ is a maximal 
independent set of vertices one has $\beta_0'(G)\leq |V_1|\leq n/2$. 
Therefore, using Corollary~\ref{vila-morey-pd} and the
Auslander-Buchsbaum formula, we get ${\rm depth}\, R/I(G)\leq
n/2$.
\end{proof}

\begin{corollary}
If\/ $G$ is a $B$-graph with $n$ vertices, then ${\rm depth}\,
R/I(G)\leq \dim R/I(G)\leq
\left\lfloor\frac{n}{2}\right\rfloor$.
\end{corollary}

\begin{proof} Recall that $n=\alpha_0(G) + \beta_0(G)$. By
  Theorem~\ref{teo1}, $\beta_0(G) \leq \alpha_0(G)$, and so
  $\beta_0(G) \leq \left\lfloor\frac{n}{2}\right\rfloor$. The result
  now follows because  
$\beta_0(G)=\dim R/I(\mathcal{C})$.
\end{proof}

Lower bounds are given in \cite{morey} for the depths of $R/I(G)^t$ for
$t\geq 1$ when $I(G)$ is the edge ideal of a tree or forest. Upper
bounds for the depth of $R/I(G)$ are given in \cite[Corollary
4.15]{bounds} when $G$ is any graph without isolated vertices. The 
depth and the Cohen-Macaulay property of ideals of mixed products is
studied in \cite{ionescu-rinaldo}. 

We close this section with an upper bound for the multiplicity of
edge rings. Let $\mathcal{C}$ be a clutter.  The {\it multiplicity\/} of the 
edge-ring $R/I(\mathcal{C})$, denoted by $e(R/I(\mathcal{C}))$, equals the number of 
faces of maximum dimension of the independence complex
$\Delta_\mathcal{C}$, i.e., the multiplicity
of $R/I(\mathcal{C})$ equals the number of independent sets of
$\mathcal{C}$ with $\beta_0(\mathcal{C})$ vertices. A related
invariant that was considered earlier is ${\rm arith}$-${\rm
  deg}(I(\mathcal{C}))$, the number of maximal independent sets of $\mathcal{C}$. 

\begin{proposition} \cite{bounds} \label{aug22-01}
If $\mathcal{C}$ is a $d$-uniform clutter and $I=I(\mathcal{C})$, then $e(R/I)\leq
d^{\alpha_0(\mathcal{C})}$.
\end{proposition}

\section{Stability of associated primes}\label{stability}

One method of gathering information
about an ideal is through its associated primes. Let $I$ be an ideal
of a ring $R$.  
In this section, we will examine the sets of
associated primes of powers of $I$, that is, the sets
$$\Ass(R/I^t)=\{{\mathfrak p}\subset R \, | \, {\mathfrak p} \ {\mbox{\rm is prime and }}
{\mathfrak p}=(I^t:c)\ 
{\mbox{\rm for some }} c\in R\}.$$ 
When $I$ is a monomial ideal of a polynomial ring $R=K[x_1,
\ldots, x_n]$, the associated primes will be monomial ideals, that is,
prime ideals which are
generated by a subset of the variables. When $I$ is a square-free
monomial ideal, the minimal primes of $I$, $\Min(R/I)$, correspond to
minimal vertex covers 
of the clutter $\mathcal{C}$ associated to $I$. 
In general $\Min(R/I)\subset \Ass(R/I^t)$ for all positive integers $t$.
For a square-free monomial ideal, in the case where equality holds for
all $t$, the ideal $I$ is said to be {\em normally
  torsion-free}. More generally, an ideal $I\subset R$ is called {\it normally
torsion-free\/} if ${\rm Ass}(R/I^i)$ 
is contained in ${\rm Ass}(R/I)$ for all $i\geq 1$ and $I\neq R$.  

In \cite{brod}, Brodmann showed that when $R$ is a Noetherian ring and $I$
is an ideal of $R$, the sets $\Ass(R/I^t)$ stabilize
for large $t$. That is, there exists a positive integer $N$ such that
$\Ass(R/I^t)=\Ass(R/I^N)$ for all $t \geq N$. We will refer to a
minimal such $N$ as the {\em index of stability} of $I$. There are two
natural questions following from this result. In this article, we will
focus on the monomial versions of the questions.

\begin{question}\label{question1}
 Given a monomial ideal $I$, what is an effective upper
  bound on the index 
  of stability for a given class of monomial ideals?
\end{question}

\begin{question}\label{question2}
 Given a monomial ideal $I$, which primes are in $\Ass(R/I^t)$
  for all sufficiently large $t$?
\end{question}

An interesting variation on
Questions~\ref{question1} and~\ref{question2} was posed in
\cite{Sharp}.

\begin{question}\label{SharpQuestion}
Suppose that $N$ is the index of stability of an ideal $I$. Given a
  prime ${\mathfrak p}\in \Ass(R/I^N)$, can you find an integer $N_{\mathfrak p}$
  for which ${\mathfrak p}\in \Ass(R/I^t)$ for $t \geq N_{\mathfrak p}$?
\end{question}

Brodmann also showed that the sets $\Ass(I^{t-1}/I^t)$ stabilize. Thus
in the general setting, one could ask similar questions about these
sets. However, for monomial ideals the following lemma shows that in
order to find information 
about $\Ass(R/I^t)$, one may instead study $\Ass(I^{t-1}/I^t)$.

\begin{lemma}
Let $I$ be a monomial ideal. Then $\Ass(I^{t-1}/I^t)=\Ass(R/I^t)$.
\end{lemma}

\begin{proof}
Suppose that ${\mathfrak p}\in \Ass(R/I^t)$. Then ${\mathfrak p}=(I^t:c)$
for some monomial $c \in R$. But since ${\mathfrak p}$ is necessarily a monomial prime,
generated by a subset of the variables, then if $xc\in I^t$ for a
variable $x\in {\mathfrak p}$, then $c\in
I^{t-1}$ and so ${\mathfrak p}\in \Ass(I^{t-1}/I^t)$. The other inclusion is
automatic. 
\end{proof}

Note that this method was used in \cite{Trung} to show that
the corresponding equality also holds for the integral closures of the
powers of $I$.

For special classes of ideals, there have been some results that use
properties of the 
ideals to find bounds on $N$. For example, if $I$ is generated by a
regular sequence, then by \cite{Hochster}, $I$ is normally
torsion-free, or $\Ass(R/I^t)=\Min(R/I)$ for all $t$, and thus
$N=1$. If instead $I$ is generated by a $d$-sequence and is strongly
Cohen-Macaulay, then it was 
shown in \cite{morey2} that $N$ is bounded above by the dimension of
the ring. In particular, $N \leq n-g+1$ where $n$ is the dimension of
the ring and $g$ is the height of the ideal. We are particularly
interested in finding similar bounds for classes of monomial ideals.

In \cite{Hoa}, Hoa used integer programming techniques to  
give an upper bound on $N$ for general monomial ideals. Let $n$ be
the number of variables, $s$ 
the number of generators of $I$, and $d$ the maximal degree of a
generator. 

\begin{theorem}{\rm\cite[Theorem 2.12]{Hoa}}
If $I$ is a monomial ideal, then the index of stability is bounded
above by
$${\rm {max}}\left\{ d(ns+s+d)(\sqrt{n})^{n+1}(\sqrt{2}d)^{(n+1)(s-1)},
    s(s+n)^4s^{n+2}d^2(2d^2)^{s^2-s+1}\right\}.$$ 
\end{theorem}

Notice that this bound can be extremely large. For general monomial
ideals, examples are given in \cite{Hoa} to show that the bound
should depend on $d$ and $n$. However, if we restrict to special 
classes of monomial ideals, much smaller bounds can be found. For
example, an alternate bound has been shown to hold for integral
closures of powers of monomial ideals.

\begin{theorem}\cite[Theorem 16]{Trung}
If $I$ is a monomial ideal, and $N_0=n2^{n-1}d^{n-2}$, then 
$\Ass(R/{\overline{I^t}})=\Ass(R/{\overline{I^{N_0}}})$ for
$t \geq N_0$ when $n \geq 2$.
\end{theorem}

Here again $n$ is the number of variables and $d$ is the maximal
degree of a generator. For the class of normal monomial ideals,
this bound on the index of stability can be significantly lower than
the general bound given above. When $n=2$, the index of stability of the
integral closures is lower still.

\begin{lemma}\cite{ME}
If $n \leq 2$, then $\Ass(R/{\overline{I^t}})=\Ass(R/{\overline{I}})$ for
all $t \geq 1$.
\end{lemma}

Note that this result is of interest for general monomial ideals;
however, when $n=2$ a square-free monomial ideal will be a complete
intersection.
Of particular interest for this article are results that use 
combinatorial and graph-theoretic properties to yield insights into
the associated primes and index of stability of monomial ideals.
One pivotal result in this area establishes a classification of all
graphs for which $N=1$. 

\begin{theorem}\label{bipartite} \cite[Theorem 5.9]{ITG} 
Let $G$ be a graph and $I$ its edge ideal. Then $G$ is bipartite if
and only if $I$ is normally torsion-free.
\end{theorem}

The result above shows that $N=1$ for the
edge ideal of a graph if and only if the graph is bipartite. Since
minimal primes correspond to minimal vertex covers, this completely
answers Questions~\ref{question1} and~\ref{question2} for bipartite
graphs. In addition if $I$ is the edge ideal
of a balanced clutter, then $N=1$ \cite{reesclu}. 

Suppose now that $G$ is a graph that is not bipartite. Then $G$
contains at least one odd cycle. For such graphs, a method of
describing embedded associated primes, and a bound
on where the stability occurs, were given in \cite{AJ}. The method of
building embedded primes centered around the odd cycles, so we first
give an alternate proof of the description of the associated primes
for this base case.

\begin{lemma}\label{cycle}
Suppose $G$ is a cycle of length $n=2k+1$ and $I$ is the
edge ideal of $G$. Then $\Ass (R/I^t)=\Min (R/I)$ if $t\leq k$ and $\Ass
(R/I^t)=\Min (R/I) \cup \{ \m \} $ if $t \geq k+1$. Moreover, when $t
\geq k+1$, $\m = (I^t:c)$ for a monomial $c$ of degree $2t-1$.
\end{lemma}

\begin{proof}
If ${\mathfrak p}\not= \m$ is a prime ideal, then $I_{\mathfrak p}$ is the
edge ideal of a bipartite graph, and thus by Theorem~\ref{bipartite}
${\mathfrak p}\in \Ass(R/I^t)$ (for any $t\geq 1$) if and only if ${\mathfrak p}$ is a minimal prime
of $I$. Notice also that the deletion of any vertex $x_i$ (which corresponds
to passing to the quotient ring $R/(x_i)$) results in a bipartite
graph as well. Thus by \cite[Corollary 3.6]{HaM}, $\m \not\in
\Ass(R/I^t)$ for $t \leq k$ since a maximal matching has $k$
edges. For $t \geq k+1$, define $b=\left( \prod_{i=1}^n x_i
\right)$ and $c=b(x_1x_2)^{t-k-1}$ where $x_1x_2$ is any edge of
  $G$. Then since 
  $c$ has degree $2t-1$, $c \not\in I^t$, but $G$ is a cycle,
  $x_ib \in I^{k+1}$ and $x_ic\in I^t$. Thus $\m =(I^t:c)$ and so
  $\m \in \Ass(R/I^t)$ for $t \geq k+1$.
\end{proof}

\begin{corollary}\label{cyclewhiskers}
Suppose $G$ is a connected graph containing an odd cycle of length
$2k+1$ and suppose that every vertex of $G$ that is not in the cycle is a
leaf. Then $\Ass(R/I^t)=\Min (R/I) \cup \{ \m \} $ if $t \geq
k+1$. Moreover, when $t \geq k+1$, $\m = (I^t:c)$ for a monomial $c$
of degree $2t-1$. 
\end{corollary}

\begin{proof}
Let $b$ and $c$ be defined as in the proof of
Lemma~\ref{cycle}. Notice that if $x$ is a leaf, then $x$ is connected
to a unique vertex in the cycle and that $xb \in I^{k+1}$. The
remainder of the proof follows as in Lemma ~\ref{cycle}.
\end{proof}

If $G$ is a more general graph, the embedded associated primes of
$I=I(G)$ are formed by working outward from the odd cycles. This was
done in \cite{AJ}, including a detailed explanation of how to work
outward from multiple odd cycles. Before providing more concise proofs of the
process, we first give an informal, but
illustrative, description. Suppose $C$ is
a cycle with $2k+1$ vertices $x_1, \ldots , x_{2k+1}$. Color the
vertices of $C$ 
red and color any noncolored vertex that is adjacent to a red vertex
blue. The set of colored vertices, together with a minimal vertex
cover of the set of edges neither of whose vertices is colored,
will be an embedded associated prime of $I^t$ for all $t \geq k+1$. To
find additional embedded 
primes of higher powers, select any blue
vertex to turn red and turn any uncolored neighbors of this vertex
blue. The set of colored vertices, together with a minimal vertex
cover of the noncolored edges,
will be an embedded associated 
prime of $I^t$ for all $t \geq k+2$. This process continues until all
vertices are colored red or blue.

The method of building new
associated primes for a power of $I$ from primes associated to lower
powers relies on localization. Since
localization will generally cause the graph (or clutter) to become
disconnected, we first need the following lemma. 

\begin{lemma}\label{disjoint}{\rm (\cite[Lemma 3.4]{HaM}, see also
  \cite[Lemma 2.1]{AJ})} 
Suppose $I$ is a square-free monomial ideal in $S = K[x_1,\dots,x_r,
  y_1, \dots, y_s]$ such that 
$I=I_1S + I_2S$, where $I_1 \subset S_1 = K[x_1, \dots, x_r]$
and $I_2 \subset S_2 = K[y_1, \dots, y_s]$. Then ${\mathfrak p}\in
\Ass (S/I^t)$ 
if and only if 
${\mathfrak p}={\mathfrak p}_1S + {\mathfrak p}_2S$, where ${\mathfrak
  p}_1 \in \Ass (S_1/I_1^{t_1})$ and ${\mathfrak p}_2 \in 
\Ass (S_2/I_2^{t_2})$ with $(t_1-1) + (t_2-1) = t-1$.
\end{lemma}

Note that this lemma easily generalizes to an ideal $I=(I_1, I_2,
\ldots, I_s)$ where the $I_i$ are edge ideals of disjoint
clutters. Then ${\mathfrak p} \in \Ass(R/I^t)$ if and only if
${\mathfrak p}=({\mathfrak p}_1, \ldots , {\mathfrak p}_s)$ with
${\mathfrak p}_i \in \Ass(R/I_i^{t_i})$ where $(t_1-1)+(t_2-1)+\cdots 
+(t_s-1)=(t-1)$.

We now fix a notation to show how to build embedded associated
primes. Consider ${\mathfrak p}\in \Ass(R/I^t)$ for $I$ the edge ideal
of a graph 
$G$. Without loss of generality, Lemma~\ref{disjoint} allows us to
assume $G$ does not have isolated vertices. If ${\mathfrak p}\not= \m$, then
since ${\mathfrak p} \in \Ass(R/I^t) \Leftrightarrow {\mathfrak
p}R_{\mathfrak p} 
\in \Ass(R_{\mathfrak p}/(I_{\mathfrak p})^t)$,
consider $I_{\mathfrak p}$. Write 
$I_{\mathfrak p}=(I_a,I_b)$ where $I_a$ is generated by all generators
of $I_{\mathfrak p}$ of
degree two and $I_b$ is the prime ideal generated by the degree one
generators of $I_{\mathfrak p}$, which correspond to the isolated vertices
of the graph associated to $I_{\mathfrak p}$. Note that the graph
corresponding to 
$I_a$ need not be connected. If $I_a=(0)$, then ${\mathfrak p}$ is a minimal
prime of $I$, so assume $I_a \not= (0)$. Define ${\mathfrak p}_a$ to be the
monomial prime generated by variables of $I_a$. Define $N_1=\cup_{x
  \in {\mathfrak p}_a}N(x)$, where $N(x)$ is the neighbor set of $x$ in $G$, and
let ${\mathfrak p}_1={\mathfrak p}_a \cup N_1$. Notice that if $x \in {\mathfrak p}_a$, then $x$ is not
isolated in $I_{\mathfrak p}$, so $N_1 \subset {\mathfrak p}$ and thus $N_1 \subset
{\mathfrak p}_a \cup I_b$. Define ${\mathfrak p}_2 = {\mathfrak p} \backslash {\mathfrak p}_1=I_b\backslash N_1$, and $N_2 =
\cup_{x \in {\mathfrak p}_1} 
N(x) \backslash {\mathfrak p}$. If $G_1$ is the induced subgraph of $G$ on the 
vertices in ${\mathfrak p}_1 \cup N_2$, $G_2$ is the induced subgraph of $G$ on
vertices in $V\backslash {\mathfrak p}_1$, and $I_i=I(G_i)$ for $i=1,2$, then
$I_{\mathfrak p}=((I_1)_{{\mathfrak p}_1}, (I_2)_{{\mathfrak p}_2})$ and ${\mathfrak p}_2$ is a minimal vertex cover
of $I_2$. By design, any vertex appearing in both $G_1$ and $G_2$ is
not in ${\mathfrak p}$, and thus $(I_1)_{{\mathfrak p}_1}$ and $(I_2)_{{\mathfrak p}_2}$ do not share a
vertex. Thus by Lemma~\ref{disjoint} and the fact that associated
primes localize, ${\mathfrak p} \in \Ass(R_{\mathfrak p}/(I_{\mathfrak p})^t)$ if and 
only if ${\mathfrak p}_1 \in \Ass(R_1/(I_1)^t)$ and ${\mathfrak p}_1$ is the maximal ideal of
$R_1=K[x\, |\, x \in {\mathfrak p}_1]$. For convenience, define $R_a=K[x\, |\, x\in {\mathfrak p}_a]$. 

\begin{proposition}\label{buildoutward}
Let ${\mathfrak p}\in \Ass(R/I^t)$. Using the notation from above, assume
${\mathfrak p}_1=(I_1^t:c)$ for some 
monomial $c\in R_a$ of degree at most $2t-1$. Let $x\in {\mathfrak
  p}_1$. Let ${\mathfrak p}_1'={\mathfrak p}_1 \cup N(x)$, and let
${\mathfrak p}_2'$ be any minimal vertex cover of the edges of $G_2'$
where $G_2'$ is the induced subgraph of $G$ on the vertices
$V\backslash {\mathfrak p}_1'$. Let $N_2'=\cup_{x \in {\mathfrak p}_1'} 
N(x) \backslash ({\mathfrak p}_1' \cup {\mathfrak p}_2')$.
Let $G_1'$ be the induced subgraph of $G$ with vertices in
${\mathfrak p}_1'\cup N_2'$. Then ${\mathfrak p}'=({\mathfrak
  p}_1',{\mathfrak p}_2') \in \Ass(R/I^{t+1})$.  
\end{proposition}

\begin{proof}
If $v$ is an isolated vertex of $G_2'$ then $N(v)
\subset {\mathfrak p}_1'$ and thus every edge of $G$ containing $v$ is covered
by ${\mathfrak p}_1'$. Hence ${\mathfrak p}'$ is a vertex cover of $G$.
Since $x\in {\mathfrak p}_1$, there is an edge $xy \in
G_1$ with $y\in {\mathfrak p}_a$. Consider $c'=cxy$. Then the 
degree of $c'$ is at most $2t+1$, so $c' \not\in (I_1)^{t+1}$. 
If $I_1'=I(G_1')$, then $c' \not\in (I_1')^{t+1}$ as well. If $z\in
{\mathfrak p}_1$, then $z(cxy)=(zc)(xy) \in (I_1')^{t+1}$. If $z \in
N(x)$, then 
$z(cxy)=(cy)(zx) \in (I_1')^{t+1}$ since $y \in {\mathfrak p}_1$. Thus ${\mathfrak p}_1'
\subset ((I_1')^{t+1}:c')$. Suppose $z
\not\in {\mathfrak p}_1'$ is a vertex of $G_1'$. Then $z \in N_2'$. Then $z \not\in
N(x)$ and $z \not\in N(y)$ since $y \in {\mathfrak p}_a$, so $zx$ and $zy$ are not
edges of $G_1'$. Also $z \not\in {\mathfrak p}_1$ and $c \in R_a$, so $zc \not\in
(I_1')^{t+1}$.  
Thus the inclusion must be an equality. 

Since ${\mathfrak p}_2'$ is a minimal vertex cover of the edges of $G_2'$, then
${\mathfrak p}_2' \in \Ass(R/I_2')$ where $I_2'$ is the edge ideal of 
$G_2'$ (where isolated vertices of $G_2'$ are not included in
$I_2'$). Note that $I_{{\mathfrak p}'}=( (I_1')_{{\mathfrak p}_1'},
(I_2')_{{\mathfrak p}_2'})$ and so the result follows
from Lemma~\ref{disjoint}. 
\end{proof}

Note that if $G$ contains an odd cycle of length $2k+1$, then embedded 
associated primes satisfying the 
hypotheses of Proposition~\ref{buildoutward} exist for $t \geq
k+1$ by Lemma~\ref{cycle} and Corollary~\ref{cyclewhiskers}. Starting
with an induced odd cycle $C$ 
one can now recover all the primes described in 
\cite[Theorem 3.3]{AJ}. In addition, combining
Corollary~\ref{cyclewhiskers} with 
Lemma~\ref{disjoint} as a starting place for
Proposition~\ref{buildoutward} recovers the result from \cite[Theorem
  3.7]{AJ} as well. Define $\Ass(R/I^t)^*$ to be the set of embedded
associate primes of $I^t$ produced in Proposition~\ref{buildoutward}
by starting from any odd cycle, or collection of odd cycles, of the
graph. Then $\Ass(R/I^t)^* \subset \Ass(R/I^s)$ for all $s \geq
t$. To see this, recall that if ${\mathfrak p}$ is not a minimal prime, then there
is a vertex $x$ such that $x \cup N(x) \subset {\mathfrak p}$. Choosing such an
$x$ results in ${\mathfrak p}_1={\mathfrak p}_1'$
and the process shows that ${\mathfrak p} \in \Ass(R/I^{t+1})$. Notice also that
the sets $\Ass(R/I^t)^*$ stabilize. In particular, 
$\Ass(R/I^t)^* = \Ass(R/I^n)^*$ for all $t\geq n$ where $n$ is the number of variables. Notice that
choosing $x\in N_1$ each time will eventually result in $\m \in
\Ass(R/I^t)^*$ for some $t$. Counting the maximal number of steps this
could take provides a bound on the index of stability. Following the
process above for a particular graph can often yield a significantly
lower power $M$ for which $\Ass(R/I^t)^*$ stabilize. 
These results are collected below.

\begin{theorem}\label{graphs}
Let $I$ be the edge ideal of a connected graph $G$ that is not
bipartite. Suppose $G$ has $n$ vertices and $s$ leaves, and $N$ is the
index of stability of $I$. 
\begin{enumerate}
\item[(a)] \cite[Theorem 4.1]{AJ} The process used in
  Proposition~\ref{buildoutward} produces all embedded associated
  primes in the stable set. That is, $\Ass(R/I^N)=\Min(R/I) \cup
  \Ass(R/I^N)^*$. 
\item[(b)] \cite[Corollary 4.3]{AJ}, (Proposition~\ref{buildoutward})
  If the smallest odd cycle of $G$ 
  has length $2k+1$, then  $N \leq n-k-s$. 
\item[(c)] \cite[Theorem 5.6, Corollary 5.7]{AJ} If $G$ has a unique
  odd cycle, then $\Ass(R/I^t)=\Min(R/I)\cup \Ass(R/I^t)^*$ for all
  $t$. Moreover, the sets $\Ass(R/I^t)$ form an ascending chain.
\item[(d)] If ${\mathfrak p}\in \Ass(R/I^N)$, and $N_0$ is the smallest positive
  integer for which ${\mathfrak p}\in \Ass(R/I^{N_0})^*$, then ${\mathfrak p}\in \Ass(R/I^t)$
  for all $t\geq N_0$. 
\end{enumerate}
\end{theorem}

To interpret Theorem~\ref{graphs} in light of our earlier questions,
notice that $(a)$ answers Question~\ref{question2}, $(b)$ answers
Question~\ref{question1}, and $(d)$ provides a good upper bound for
$N_{\mathfrak p}$ in Question~\ref{SharpQuestion}. 
The significance of $(c)$ is to answer a fourth question of
interest. Before presenting that question, we first discuss some
extensions of the above results to graphs containing loops. 

\begin{corollary}\label{loops}
Let $I$ be a monomial ideal, not necessarily square-free, such that
the generators of $I$ have degree at most two. Define $\Ass(R/I^t)^*$
to be the set of embedded 
associate primes of $I^t$ produced in Proposition~\ref{buildoutward}
by starting from any odd cycle, or collection of odd cycles where
generators of $I$ that are not square-free are considered to be cycles
of length one. Then the results of Theorem~\ref{graphs} hold for $I$.
\end{corollary} 

\begin{proof}
If $I$ has generators of degree one, then write $I=(I_1,I_2)$ where
$I_2$ is generated in degree one. Then $I_2$ is a complete
intersection, so by using Lemma~\ref{disjoint} we may replace $I$ by
$I_1$ and assume $I$ is generated in degree two. If $I$ is not
square-free, consider a generator $x^2 \in I$. This 
generator can be represented as a loop (cycle of length one) in the
graph. Define ${\mathfrak p}_a=(x)$ and $N_1=N(x)$. Note that ${\mathfrak p}_1={\mathfrak p}_a \cup N_1 =
(I_1 :c)$ where $I_1$ is the induced graph on $x \cup N(x)$ and
$c=x$. Then ${\mathfrak p}_1$ satisfies the hypotheses of
Proposition~\ref{buildoutward}. The results now follow from the proof
of Proposition~\ref{buildoutward}.
\end{proof}

Notice that ideals that are not square-free will generally have
embedded primes starting with $t=1$ since the smallest odd cycle has
length $1=2(0)+1$, so $k+1=1$. The above corollary can be extended to
allow for any pure powers of variables to be generators of the ideal
$I$.

\begin{corollary}
Let $I=(I_1,I_2)$ where $I_2$ is the edge ideal of a graph $G$ and
$I_1=(x_{i_1}^{s_1}, \ldots , x_{i_r}^{s_r})$ for any powers $s_j \geq
1$. Then the results of Theorem~\ref{graphs} hold for $I$ with
$\Ass(R/I^t)^*$ defined as in Corollary~\ref{loops}.
\end{corollary}

\begin{proof}
As before, we may assume $x_j \geq 2$ for all $j$. Let
$K=(x_{i_1}^{2}, \ldots , x_{i_r}^{2})$ and let 
$J=(K,I_2)$. Then $J$ satisfies the hypotheses of
Corollary~\ref{loops}. Let ${\mathfrak p}\in \Ass(R/J^t)^*$ be formed by starting
with ${\mathfrak p}_a=(x_{i_1})$ and let
${\mathfrak p}_1=(J_1^t:c)$ where $J_1$ and $c$ are defined as in
Corollary~\ref{loops}. Suppose $x_{i_1}, \ldots , x_{i_v}\in {\mathfrak p}_1$. Let
$q_{j}\geq  
0$ be the least integers such that $c'=x_{i_1}^{q_1}\cdots
x_{i_v}^{q_v}\cdot c \not\in I_1^t$. Then it
is straightforward to check that ${\mathfrak p}_1=(I_1^t:c')$ and so ${\mathfrak p} \in
\Ass(R/I^t)^*$. Thus higher powers of variables can also be treated as
loops and the results of Theorem~\ref{graphs} hold.
\end{proof}

We now return to Theorem~\ref{graphs} $(c)$. In general, the sets
$\Ass(R/I^t)^*$ form an ascending chain. Theorem~\ref{graphs} $(c)$
gives a class of graphs for which $\Ass(R/I^t)^*$ describe every
embedded prime of a power of $I^t$ optimally. Thus $\Ass(R/I^t)$ will
form a chain. This happens for many classes of monomial ideals, and
leads to the fourth question.

\begin{question}\label{chainconjecture}
If $I$ is a square-free monomial ideal, is $\Ass(R/I^t) \subset
\Ass(R/I^{t+1})$ for all $t$?
\end{question}

For monomial ideals, Question~\ref{chainconjecture} is of interest for low
powers of $I$. For sufficiently large powers, the sets of associated
primes are known to form an ascending chain, and a bound beyond which
the sets $\Ass(I^t/I^{t+1})$ form a chain has been shown by multiple
authors (see \cite{Sharp}, \cite{ME}). This bound depends on two
graded algebras which encode information on the powers of $I$, and
which will prove useful in other results. The first is the {\em Rees
  algebra} $R[It]$ of $I$, which is defined by
$$R[It] = R \oplus It \oplus I^2t^2 \oplus I^3t^3 \oplus \cdots$$
and the second is the {\em associated graded ring} of $I$,
$${\rm gr}_I(R) = R/I \oplus I/I^2 \oplus I^2/I^3 \oplus \cdots .$$
Notice that while the result is for $\Ass(I^t/I^{t+1})$, for monomial
ideals $\Ass(R/I^{t+1})$ will also form a chain. 

\begin{theorem} \cite{Sharp, ME}
$\Ass(I^t/I^{t+1})$ is increasing for $t > a_{R[It]_+}^0({\rm gr}_I(R))$. 
\end{theorem}

Note that square-free is essential in
Question~\ref{chainconjecture}. Examples of monomial 
ideals for which the associated primes do not form an ascending chain
have been given in \cite{HH, Hoa}. Those examples were designed for
other purposes and so are more complex than what is needed here. A
simple example can be found by taking the product of 
consecutive edges of an odd cycle.

\begin{example}\label{notchain}
Let $I=(x_1x_2^2x_3,x_2x_3^2x_4,x_3x_4^2x_5,x_4x_5^2x_1,x_5x_1^2x_2).$
If $\mathfrak{m}=(x_1,x_2,x_3,x_4,x_5)$, then $\mathfrak{m}\in
\Ass(R/I^t)$ for $t=1,4$, but $\mathfrak{m} \not\in \Ass(R/I^t)$ for $t=2,3$. 
\end{example}

The ideal in Example~\ref{notchain} can be viewed as multiplying
adjacent edges in a $5$-cycle to form generators of $I$, and so has
a simple combinatorial realization. A similar result holds for longer
odd cycles, where the maximal ideal is not the only associate prime to
appear and disappear. However, if instead
$I=(x_1x_2x_3,x_2x_3x_4,x_3x_4x_5,x_4x_5x_1,x_5x_1x_2)$ is the path
ideal of the pentagon, then $\Ass(R/I^t)=\Min(R/I)\cup\{\m\}$ for $t\geq
2$ and thus $\Ass(R/I^t)$ form an ascending chain (see \cite[Example
  3.14]{HaM}). 

There are some interesting cases where associated primes are known to
form ascending chains. The first listed is quite general, but has
applications to square-free monomial ideals.

\begin{theorem}\label{integralclosurechain} {\rm(\cite[Proposition
      3.9]{McAdam}, see also 
    \cite[Proposition 16.3]{HIO})} 
If $R$ is a Noetherian ring, then $\Ass(R/{\overline{I^t}})$ form an
ascending chain.
\end{theorem}

In order to present the next class of ideals for which the associated
primes are known to form ascending chains, we first need some some
background definitions. 

\begin{definition} Let $G$ be a graph. A {\it colouring\/} of the
vertices of $G$ is an assignment
of colours to the vertices of $G$ such that adjacent vertices
have distinct colours. The {\it chromatic number\/} of $G$ is the
minimal number of colours in a colouring of $G$. A graph is called
{\it perfect\/} if for every induced subgraph $H$, the
chromatic number of $H$ equals the size of the largest complete subgraph
of $H$.
\end{definition}

An excellent reference for the theory of perfect graphs is the book
of Golumbic \cite{golumbic}. Using perfect graphs, we now give an
example to show how Theorem~\ref{integralclosurechain} can be applied
to classes of square-free monomial ideals. An 
alternate proof appears in \cite[Corollary 5.11]{FHV}. 

\begin{example}\label{perfectgraph}
If $I$ is the ideal of minimal vertex covers of a perfect graph, then
$\Ass(R/I^t)$ form an ascending chain.
\end{example}

\begin{proof}
By \cite[Theorem 2.10]{perfect}, $R[It]$ is normal. Thus
$I^t={\overline{I^t}}$ for all $t$, so by
Theorem~\ref{integralclosurechain}, $\Ass(R/I^t)$ form an ascending
chain.
\end{proof}

Similar results hold for other classes of monomial ideals for which
$R[It]$ is known to be normal. For example, in \cite[Corollary
4.2]{ConcaDeNegri} it is shown that a path ideal of a rooted tree has
  a normal Rees algebra. Thus by Theorem~\ref{integralclosurechain},
  $\Ass(R/I^t)$ form an ascending chain for such ideals. Note that a
  path ideal can be viewed as the edge ideal of a carefully chosen
  uniform clutter.

It is interesting to compare the result of Example~\ref{perfectgraph}
 to \cite[Theorem 5.9]{FHV}, where it is shown that if $I$ is the
 ideal of minimal vertex covers of a perfect graph, then the set of
 primes associated to any fixed power has the saturated chain property. Here
 $\Ass(R/I^t)$ has the saturated chain property if for 
 every ${\mathfrak p}\in \Ass(R/I^t)$, either ${\mathfrak p}$ is minimal or there is a $Q
 \subsetneq {\mathfrak p}$ with $Q \in \Ass(R/I^t)$ and $\height Q = \height {\mathfrak p}-1$.

Notice that Theorem~\ref{graphs} shows that in the case of a graph
with a unique odd cycle, Question~\ref{chainconjecture} has an
affirmative answer. This
result can be generalized to any graph containing a leaf.
First we need a slight variation of a previously known
result. 

\begin{lemma}\cite[Lemma 2.3]{morey2}\label{depth1argument}
Suppose $I=I(G)$ is the edge ideal of a graph and $a\in I/I^2$ is a
regular element of the associated graded ring ${\rm gr}_I(R)$. Then the sets
$\Ass(R/I^t)$ form an ascending chain. Moreover,
$\Ass(R/I^t)=\Ass(I^{t-1}/I^t)$ for all $t\geq 1$.
\end{lemma}

\begin{proof}
Let $a \in I/I^2$ be a regular element of ${\rm gr}_I(R)$. Assume ${\mathfrak p} \in
\Ass(I^t/I^{t+1})$. Then there is a $c \in I^t/I^{t+1}$ with
${\mathfrak p}=(0:_{R/I} c)$. But then ${\mathfrak p}=(0:_{R/I} ac)$, and $a$ lives in degree
  one, so ${\mathfrak p}\in \Ass(I^{t+1}/I^{t+2})$. So these sets form an
  ascending chain. Now the standard short exact sequence
$$0 \rightarrow I^t/I^{t+1} \rightarrow R/I^{t+1} \rightarrow R/I^t \rightarrow 0$$
 gives 
$$\Ass(I^t/I^{t+1}) \subset \Ass(R/I^{t+1}) \subset \Ass(R/I^t)
  \cup \Ass(I^t/I^{t+1})$$
and the result follows by induction.
\end{proof}

\begin{proposition}\label{leaf}
Let $G$ be a graph containing a leaf $x$ and let $I=I(G)$ be its edge
ideal. Then $\Ass(R/I^t) \subset \Ass(R/I^{t+1})$ for all $t$. That
is, the sets of associated primes of the powers of $I$ form an
ascending chain.
\end{proposition}

\begin{proof}
Since $x$ is a leaf of $G$, there is a unique generator $e=xy \in I$
divisible by $x$. Let $a$ denote the image of $e$ in $I/I^2$. We claim
$a$ is a regular element of ${\rm gr}_I(R)$. To see this, it suffices to
show that if $fe\in I^{t+1}$ for some $t$, then $f \in I^t$. 
Since $I$
is a monomial ideal and $e$ is a monomial, $fe \in I^{t+1}$ if and only
  if every term of $fe$ is in $I^{t+1}$. Thus we may assume $f$ is a
  monomial and $fxy=e_1e_2\cdots e_{t+1}h$ for some edges $e_i$ of $G$
  and some monomial $h$. Suppose $f \not\in I^t$. Then $x$ divides
  $e_i$ for some $i$, say $i=t+1$. Since  
  $x$ is a leaf, $e_i=xy$ and by cancellation $f=e_1\cdots e_th \in
  I^t$. Thus $a$ is a regular element of ${\rm gr}_I(R)$ and by
  Lemma~\ref{depth1argument} the result follows.
\end{proof}

When extending the above results to more general square-free monomial
ideals, one needs to pass from graphs to clutters. An obstruction to
extending the results is the lack of an analog to
Theorem~\ref{bipartite} (\cite[Theorem
  5.9]{ITG}). One possible analog appears as a conjecture of
Conforti and Cornu\'ejols, see \cite[Conjecture 1.6]{cornu-book},
which we discuss later in this section. This
conjecture is stated in the language of combinatorial optimization. It
says that a clutter $\C$ has the {\it max-flow min-cut} (MFMC, see
Definition~\ref{mfmc-def})
property if and only if $\C$ has the {\it packing} property. These
criterion for clutters have been shown in recent years to have
algebraic translations \cite{reesclu} which will be discussed in
greater detail later in
the section. An ideal satisfies the packing
property if the monomial grade of $I$ (see
Definition~\ref{mon-grade}) 
is equal to the height of $I$
and this same equality holds for every minor of $I$ \cite{cornu-book,reesclu}. 
Here a minor is
formed by either localizing at a collection of variables, passing to
the image of $I$ in a quotient ring $R/(x_{i_1}, x_{i_2}, \ldots ,
x_{i_s})$, or a combination of the two. In \cite[Corollary
  3.14]{clutters} and \cite[Corollary 1.6]{HHTZ}, it was shown that $\C$
satisfies MFMC if and only if the corresponding edge ideal 
$I(\C)$ is normally torsion-free. 
This allows the conjecture to be
restated (cf. 
\cite[Conjecture 4.18]{reesclu}) as: if $\C$ has the packing property, then
$I(\C)$ is normally torsion-free.

Since a proof of this conjecture does not yet exist, the techniques
used to describe the embedded associated primes, the stable set of
associated primes, and the index of stability for graphs are difficult
to extend. However some partial results are known. The first gives
some conditions under which it is known that the maximal ideal is, or
is not, an associated prime. In special cases, this can provide a seed
for additional embedded associated primes using techniques such as
those in Proposition~\ref{buildoutward}.

\begin{theorem}\label{family-cc}
If $I$ is a square-free monomial ideal, every proper
  minor of $I$ is normally torsion-free, and $\beta_1$ is the monomial
  grade of $I$, then 
\begin{enumerate} 
\item[(a)] \cite[Corollary3.6]{HaM} $\m \not\in \Ass(R/I^t)$ for $t \leq \beta_1$. 
\item[(b)] \cite[Theorem 4.6]{HaM} If $I$ fails the packing
  property, then $\m \in \Ass(R/I^{\beta_1+1})$.  
\item[(c)] \cite[Proposition 3.9]{HaM} If $I$ is unmixed and
  satisfies the packing property, then $I$ 
is normally torsion-free.
\end{enumerate}
\end{theorem}

Other recent results have taken a different approach. Instead of
working directly with the edge ideal of a clutter $\C$, one can work
with its Alexander dual, which is again the edge ideal of a
clutter. Using this approach, the embedded associated primes of the
Alexander dual have been linked to colorings of a clutter.
Recall that $\chi(\C)$ is the minimal number $d$ for which there is a
partition $X_1, \ldots , X_d$ of the vertices of $\C$ for which for
all edges $f$ of $\C$, $f
\not\subset X_i$ for every $i$. A clutter is {\it critically}
$d$-{\it chromatic} if $\chi(\C)=d$ but $\chi(\C\backslash \{x\}) < d$ for
every vertex $x$.  

\begin{theorem}
\begin{enumerate}
\item[(a)]\cite[Corollary 4.6]{FHV} If $I$ is the ideal of covers of
  a clutter $\C$, and if the induced subclutter ${\C}_{\mathfrak p}$ on the vertices
  in ${\mathfrak p}$ is critically $(d+1)$-chromatic, then ${\mathfrak p} \in \Ass(R/I^d)$ but
  ${\mathfrak p} \not\in \Ass(R/I^t)$ for any $t\leq d-1$.
\item[(b)] \cite[Theorem 5.9]{FHV} If $I$ is the ideal of covers of a
  perfect graph $G$, then ${\mathfrak p} \in \Ass(R/I^t)$ if and only if the
  induced graph on the vertices in ${\mathfrak p}$ is a clique of size at most $t+1$.
\end{enumerate}
\end{theorem}

If one restricts to a particular power, additional results on embedded
associate primes are known. For example, in \cite[Corollary
  3.4]{FHVodd} it is shown that if $I$ is the edge ideal of the Alexander dual of a
graph $G$, then embedded primes of $R/I^2$ are in one-to-one
correspondence with induced odd cycles of $G$. More precisely, ${\mathfrak p} \in
\Ass(R/I^2)$ is an embedded prime if and only if the induced subgraph
of $G$ on the vertices in ${\mathfrak p}$ is an induced odd cycle of $G$.

An interesting class of ideals, which is in a sense dual to the edge
ideals of graphs, is unmixed square-free monomial ideals of height
two. These are the Alexander duals of edge ideals of graphs, which can
be viewed as edge ideals of clutters where, instead of requiring
that each edge has two vertices, it is instead required that each
minimal vertex cover has two vertices. For such ideals it has been
shown in \cite[Theorem 1.2]{FranciscoHaVanTuyl} that an affirmative
answer to a conjecture on graph colorings, \cite[Conjecture
  1.1]{FranciscoHaVanTuyl}, would imply an affirmative answer to
Question~\ref{chainconjecture}. In \cite[Corollary
  3.11]{FranciscoHaVanTuyl} it is shown that this conjecture holds for cliques,
odd holes, and odd antiholes. Thus the Alexander duals of these
special classes of graphs provide additional examples where
Question~\ref{chainconjecture} has an affirmative answer.

%%%%%%%%%%%%%%%%%%%%%%%%%%%%%

We now provide a more detailed discussion of the Conforti-Cornu\'ejols
conjecture, followed by a collection of results which provide families
of clutters where the conjecture is known to be true (such a
family was already given in Theorem~\ref{family-cc}(c)). We also
discuss some algebraic versions of this conjecture and how it relates
to the depth of powers of edge ideals and to normality and
torsion-freeness.

Having defined the notion of a {\it minor} for edge ideals, using the
correpondence between clutters and square-free monomias
ideals, we also have the notion of a {\it minor} of a clutter. We say
that $\mathcal{C}$ has the {\it packing property\/} if $I(\mathcal{C})$ has
this property. 

\begin{definition} Let $A$ be the incidence matrix of a clutter
$\mathcal{C}$. The {\it 
set covering polyhedron\/} is the rational polyhedron: 
\[
Q(A)=\{x \in \mathbb{R}^n  \, | \; x \geq \mathbf{0}, \; xA\geq
\mathbf{1} \},
\]
where $\mathbf{0}$ and
$\mathbf{1}$ are vectors whose entries
are equal to $0$ and $1$ respectively. Often we denote the vectors 
$\mathbf{0}$, $\mathbf{1}$ simply by $0$, $1$. We say that $Q(A)$ is 
{\it integral} if it has only integral vertices.
\end{definition}

\begin{theorem}{\rm (A. Lehman \cite{lehman}, 
\cite[Theorem~1.8]{cornu-book})}\label{lehman} If\, a clutter
  $\mathcal C$ has the 
packing property, then $Q(A)$ is 
integral. 
\end{theorem}

%\demo It follows from \cite{lehman}. 
%See also \cite[Theorem~1.8]{cornu-book}. \QED 

%\label{oct21-03} 

The converse is not true. A famous example is the clutter
$\mathcal{Q}_6$, given below. It does not pack and $Q(A)$ is integral.  

\begin{example}\label{q6}\rm Let $I=(x_1x_2x_5,x_1x_3x_4,x_2x_3x_6,x_4x_5x_6)$.
The figure:

\bigskip

\bigskip

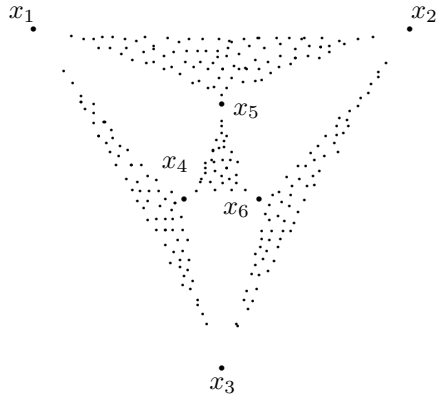
\begin{small}
\special{em:linewidth 0.4pt} \unitlength 0.5mm \linethickness{0.4pt}
\begin{picture}(30,100)(-50,30)
%\emline{30.00}{130.00}{1}{130.00}{130.00}{2}
\put(80,40){\circle*{1.5}} \put(30,130){\circle*{1.5}}
\put(70,85){\circle*{1.5}} \put(90,85){\circle*{1.5}}
\put(80,110){\circle*{1.5}} \put(130,130){\circle*{1.5}}
%\emline{130.00}{130.00}{3}{80.00}{40.00}{4}
%\emline{80.00}{40.00}{5}{30.00}{130.00}{6}
%\emline{70.00}{85.00}{7}{90.00}{85.00}{8}
%\emline{90.00}{85.00}{9}{79.67}{110.00}{10}
%\emline{79.67}{110.00}{11}{70.33}{85.00}{12}
%\emline{70.33}{85.00}{13}{80.00}{40.00}{14}
%\emline{80.00}{40.00}{15}{90.00}{85.00}{16}
%\emline{90.00}{85.00}{17}{130.00}{130.50}{18}
%\emline{130.00}{130.00}{19}{79.50}{110.00}{20}
%\emline{80.00}{110.00}{21}{30.00}{130.00}{22}
%\emline{30.00}{130.00}{23}{70.00}{85.00}{24}
\put(67.33,124.00){\circle*{.1}} \put(73.67,124.33){\circle*{.1}}
\put(75.83,124.33){\circle*{.1}} \put(84.67,123.33){\circle*{.1}}
\put(91.33,125.00){\circle*{.1}} \put(90.00,127.67){\circle*{.1}}
\put(89.67,127.67){\circle*{.1}} \put(80.33,119.33){\circle*{.1}}
\put(80.33,119.33){\circle*{.1}} \put(92.33,121.33){\circle*{.1}}
\put(106.33,125.33){\circle*{.1}}
\put(106.33,125.33){\circle*{.1}} \put(99.67,124.67){\circle*{.1}}
\put(99.67,122.33){\circle*{.1}} \put(99.67,122.33){\circle*{.1}}
\put(102.00,127.67){\circle*{.1}}
\put(111.33,127.00){\circle*{.1}}
\put(101.33,127.33){\circle*{.1}} \put(85.67,116.00){\circle*{.1}}
\put(57.33,127.00){\circle*{.1}} \put(46.00,127.33){\circle*{.1}}
\put(50.33,125.00){\circle*{.1}} \put(56.33,122.00){\circle*{.1}}
\put(59.67,128.00){\circle*{.1}} \put(62.00,121.33){\circle*{.1}}
\put(69.00,118.00){\circle*{.1}} \put(78.00,114.67){\circle*{.1}}
\put(71.33,127.33){\circle*{.1}} \put(82.00,127.00){\circle*{.1}}
\put(84.67,119.33){\circle*{.1}} \put(74.00,116.33){\circle*{.1}}
\put(86.33,119.67){\circle*{.1}} \put(62.00,124.67){\circle*{.1}}
\put(116.33,127.33){\circle*{.1}}
\put(103.00,123.00){\circle*{.1}} \put(95.67,123.00){\circle*{.1}}
\put(94.67,127.67){\circle*{.1}} \put(95.33,120.00){\circle*{.1}}
\put(86.67,119.33){\circle*{.1}} \put(90.67,116.67){\circle*{.1}}
\put(81.33,114.33){\circle*{.1}} \put(82.67,126.67){\circle*{.1}}
\put(62.00,125.67){\circle*{.1}} \put(65.67,127.67){\circle*{.1}}
\put(51.00,128.00){\circle*{.1}} \put(40.00,128.33){\circle*{.1}}
\put(38.00,119.00){\circle*{.1}} \put(40.33,115.67){\circle*{.1}}
\put(42.67,112.33){\circle*{.1}} \put(42.67,110.67){\circle*{.1}}
\put(42.67,110.67){\circle*{.1}} \put(46.00,110.00){\circle*{.1}}
\put(43.67,108.67){\circle*{.1}} \put(48.67,105.33){\circle*{.1}}
\put(48.83,105.33){\circle*{.1}} \put(49.00,105.33){\circle*{.1}}
\put(46.00,105.00){\circle*{.1}} \put(47.33,101.67){\circle*{.1}}
\put(47.33,101.67){\circle*{.1}} \put(51.00,101.00){\circle*{.1}}
\put(50.33,97.33){\circle*{.1}} \put(53.00,99.00){\circle*{.1}}
\put(55.00,98.00){\circle*{.1}} \put(52.67,96.33){\circle*{.1}}
\put(52.00,94.00){\circle*{.1}} \put(56.00,94.00){\circle*{.1}}
\put(56.00,94.00){\circle*{.1}} \put(60.33,91.00){\circle*{.1}}
\put(57.67,89.33){\circle*{.1}} \put(58.33,87.00){\circle*{.1}}
\put(60.67,87.00){\circle*{.1}} \put(63.33,87.00){\circle*{.1}}
\put(66.67,86.67){\circle*{.1}} \put(67.67,83.67){\circle*{.1}}
\put(67.67,83.67){\circle*{.1}} \put(63.67,83.00){\circle*{.1}}
\put(59.33,82.33){\circle*{.1}} \put(54.67,90.33){\circle*{.1}}
\put(60.33,84.33){\circle*{.1}} \put(60.33,79.33){\circle*{.1}}
\put(62.67,79.67){\circle*{.1}} \put(65.00,79.67){\circle*{.1}}
\put(67.00,79.67){\circle*{.1}} \put(69.33,80.00){\circle*{.1}}
\put(69.67,76.67){\circle*{.1}} \put(67.00,76.67){\circle*{.1}}
\put(65.00,76.67){\circle*{.1}} \put(65.00,76.67){\circle*{.1}}
\put(62.67,76.67){\circle*{.1}} \put(65.00,72.67){\circle*{.1}}
\put(67.33,70.67){\circle*{.1}} \put(69.67,71.00){\circle*{.1}}
\put(69.00,68.67){\circle*{.1}} \put(69.00,66.67){\circle*{.1}}
\put(71.00,66.00){\circle*{.1}} \put(70.00,64.33){\circle*{.1}}
\put(72.33,63.00){\circle*{.1}} \put(71.67,61.00){\circle*{.1}}
\put(73.67,58.00){\circle*{.1}} \put(73.67,56.67){\circle*{.1}}
\put(75.00,54.33){\circle*{.1}} \put(76.00,51.67){\circle*{.1}}
\put(56.33,87.33){\circle*{.1}} \put(58.67,94.00){\circle*{.1}}
\put(84.33,51.00){\circle*{.1}} \put(84.00,51.67){\circle*{.1}}
\put(87.00,55.67){\circle*{.1}} \put(86.00,58.00){\circle*{.1}}
\put(88.33,57.67){\circle*{.1}} \put(86.67,60.33){\circle*{.1}}
\put(66.67,120.00){\circle*{.1}} \put(54.67,125.00){\circle*{.1}}
\put(106.67,127.67){\circle*{.1}}
\put(108.67,124.00){\circle*{.1}} \put(81.33,123.00){\circle*{.1}}
\put(87.33,125.67){\circle*{.1}} \put(89.00,122.00){\circle*{.1}}
\put(75.00,127.33){\circle*{.1}} \put(82.67,117.00){\circle*{.1}}
\put(92.67,83.33){\circle*{.1}} \put(92.67,79.33){\circle*{.1}}

\put(98.00,83.00){\circle*{.1}} \put(88.00,65.00){\circle*{.1}}
\put(90.00,64.67){\circle*{.1}} \put(92.33,65.67){\circle*{.1}}
\put(94.00,68.33){\circle*{.1}} \put(91.00,69.00){\circle*{.1}}
\put(89.00,70.33){\circle*{.1}} \put(92.00,71.00){\circle*{.1}}
\put(95.00,72.00){\circle*{.1}} \put(92.67,74.00){\circle*{.1}}
\put(90.67,74.33){\circle*{.1}} \put(90.67,76.33){\circle*{.1}}
\put(93.33,76.33){\circle*{.1}} \put(96.33,76.33){\circle*{.1}}
\put(98.33,76.33){\circle*{.1}} \put(99.33,78.33){\circle*{.1}}
\put(96.67,79.00){\circle*{.1}} \put(95.00,80.33){\circle*{.1}}
\put(91.67,81.00){\circle*{.1}} \put(96.00,84.00){\circle*{.1}}
\put(102.00,84.67){\circle*{.1}} \put(98.33,86.00){\circle*{.1}}
\put(95.33,86.67){\circle*{.1}} \put(98.67,88.33){\circle*{.1}}
\put(102.67,88.33){\circle*{.1}} \put(104.67,88.67){\circle*{.1}}
\put(106.00,92.33){\circle*{.1}} \put(102.33,93.67){\circle*{.1}}
\put(100.33,93.67){\circle*{.1}} \put(104.67,95.33){\circle*{.1}}
\put(107.33,96.33){\circle*{.1}} \put(111.00,97.33){\circle*{.1}}
\put(107.67,101.67){\circle*{.1}}
\put(110.67,102.00){\circle*{.1}}
\put(112.33,102.33){\circle*{.1}}
\put(113.67,104.33){\circle*{.1}}
\put(113.33,106.00){\circle*{.1}}
\put(111.33,106.33){\circle*{.1}}
\put(114.00,108.33){\circle*{.1}}
\put(116.33,109.67){\circle*{.1}}
\put(117.00,111.67){\circle*{.1}}
\put(119.00,113.00){\circle*{.1}}
\put(119.67,115.33){\circle*{.1}}
\put(121.67,117.33){\circle*{.1}}
\put(121.67,118.67){\circle*{.1}}
\put(123.67,121.00){\circle*{.1}}
\put(120.33,128.00){\circle*{.1}}
\put(112.67,126.00){\circle*{.1}} \put(86.00,127.33){\circle*{.1}}
\put(70.33,121.33){\circle*{.1}} \put(72.67,120.00){\circle*{.1}}
\put(77.67,121.00){\circle*{.1}} \put(68.67,126.67){\circle*{.1}}
\put(59.33,124.00){\circle*{.1}} \put(68.33,74.00){\circle*{.1}}
\put(62.67,90.00){\circle*{.1}} \put(73.33,87.00){\circle*{.1}}
\put(77.00,89.33){\circle*{.1}} \put(80.00,91.00){\circle*{.1}}
\put(80.33,89.00){\circle*{.1}} \put(84.00,89.00){\circle*{.1}}
\put(84.00,91.33){\circle*{.1}} \put(86.33,87.33){\circle*{.1}}
\put(82.33,95.33){\circle*{.1}} \put(80.33,95.33){\circle*{.1}}
\put(77.33,95.33){\circle*{.1}} \put(79.33,98.33){\circle*{.1}}
\put(80.67,98.67){\circle*{.1}} \put(80.67,100.33){\circle*{.1}}
\put(78.67,100.67){\circle*{.1}} \put(80.00,102.33){\circle*{.1}}
\put(80.00,104.33){\circle*{.1}} \put(80.00,105.67){\circle*{.1}}
\put(95.67,126.00){\circle*{.1}} \put(89.33,61.00){\circle*{.1}}
\put(89.00,67.33){\circle*{.1}} \put(90.33,72.00){\circle*{.1}}
\put(95.67,74.33){\circle*{.1}} \put(100.33,81.33){\circle*{.1}}
\put(93.67,86.00){\circle*{.1}} \put(97.33,90.67){\circle*{.1}}
\put(100.67,90.33){\circle*{.1}} \put(104.33,91.00){\circle*{.1}}
\put(110.00,99.67){\circle*{.1}} \put(105.67,99.00){\circle*{.1}}
\put(103.67,97.00){\circle*{.1}} \put(110.67,104.33){\circle*{.1}}
\put(66.33,67.67){\circle*{.1}} \put(69.67,62.00){\circle*{.1}}
\put(71.33,68.67){\circle*{.1}} \put(71.00,74.00){\circle*{.1}}
\put(65.00,85.00){\circle*{.1}} \put(57.67,84.33){\circle*{.1}}
\put(66.67,82.00){\circle*{.1}} \put(49.33,99.67){\circle*{.1}}
\put(50.67,103.33){\circle*{.1}} \put(46.00,107.67){\circle*{.1}}
\put(60.00,121.00){\circle*{.1}} \put(64.67,119.33){\circle*{.1}}
\put(72.33,115.67){\circle*{.1}} \put(76.33,114.33){\circle*{.1}}
\put(79.33,117.00){\circle*{.1}} \put(76.33,119.33){\circle*{.1}}
\put(74.33,121.33){\circle*{.1}} \put(70.67,124.33){\circle*{.1}}
\put(64.67,123.00){\circle*{.1}} \put(47.33,125.00){\circle*{.1}}
\put(48.33,128.00){\circle*{.1}} \put(43.33,127.33){\circle*{.1}}
\put(54.33,127.33){\circle*{.1}} \put(98.00,127.67){\circle*{.1}}
\put(90.00,119.33){\circle*{.1}} \put(89.00,116.33){\circle*{.1}}
\put(93.00,118.33){\circle*{.1}} \put(98.33,120.67){\circle*{.1}}
\put(106.00,122.33){\circle*{.1}}
\put(103.33,125.67){\circle*{.1}}
\put(108.67,125.67){\circle*{.1}} \put(78.33,87.33){\circle*{.1}}
\put(82.67,87.33){\circle*{.1}} \put(82.00,90.33){\circle*{.1}}
\put(84.33,94.67){\circle*{.1}} \put(81.33,93.00){\circle*{.1}}
\put(81.33,93.00){\circle*{.1}} \put(75.67,91.67){\circle*{.1}}
\put(74.33,89.33){\circle*{.1}} \put(77.67,97.67){\circle*{.1}}
\put(82.67,97.33){\circle*{.1}} \put(81.33,101.67){\circle*{.1}}
\put(79.33,127.33){\circle*{.1}} \put(79.00,124.33){\circle*{.1}}
\put(81.67,125.00){\circle*{.1}} \put(79.00,122.67){\circle*{.1}}
\put(83.33,120.33){\circle*{.1}} \put(53.33,123.33){\circle*{.1}}
\put(80.00,112.67){\circle*{.1}} \put(78.33,92.33){\circle*{.1}}
\put(74.67,90.33){\circle*{.1}} \put(77.67,98.00){\circle*{.1}}
\put(75.67,93.00){\circle*{.1}} \put(78.00,94.67){\circle*{.1}}
\put(75.67,94.33){\circle*{.1}}
\put(27.00,134.00){\makebox(0,0)[cc]{\small $x_1$}}
\put(134.00,134.00){\makebox(0,0)[cc]{\small $x_2$}}
\put(80.33,35.33){\makebox(0,0)[cc]{\small $x_3$}}
\put(86.67,108.33){\makebox(0,0)[cc]{\small $x_5$}}
\put(84.00,82.00){\makebox(0,0)[cc]{\small $x_6$}}
\put(67.67,94.00){\makebox(0,0)[cc]{\small $x_4$}}
\end{picture}
\end{small}

\noindent  corresponds to the clutter associated to $I$. This clutter
will be denoted by ${\mathcal 
Q}_6$. Using {\it Normaliz\/} \cite{normaliz2} we obtain that
$\overline{R[It]}= 
R[It][x_1\cdots x_6t^2]$. Thus $R[It]$ is not normal. An interesting
property of this example is that ${\rm Ass}(R/\overline{I^i})={\rm
Ass}(R/I)$ for all $i$ (see \cite{reesclu}). 
\end{example}

\begin{definition}\label{mon-grade}
A set of edges of a clutter $\mathcal C$ is {\it
independent\/} if no two of them have a
common vertex.  We denote the maximum number of independent edges of ${\mathcal C}$ by 
$\beta_1({\mathcal C})$. We call $\beta_1({\mathcal C})$ the {\it edge independence
number\/} of $\mathcal{C}$ or the {\it monomial grade} of $I$. 
\end{definition}

Let $A$ be the incidence matrix of $\mathcal{C}$. The
edge independence number and the covering number are 
related to 
min-max problems because they satisfy: 
\begin{eqnarray*}
\lefteqn{\alpha_0({\mathcal C})\geq {\rm min}\{\langle {1},x\rangle \vert\, x\geq 0; xA\geq 
{1}\}}\\
&\ \ \ \ \ \ \ \ \ \ &
={\rm max}\{\langle y,{1}\rangle \vert\, y\geq 0; Ay\leq{1}\}
\geq \beta_1({\mathcal C}). 
\end{eqnarray*}
Notice that $\alpha_0({\mathcal C})=\beta_1({\mathcal C})$ if and only if both sides of 
the equality have integral optimum solutions. 

\begin{definition}\label{mfmc-def}\rm A clutter $\mathcal C$, with
incidence matrix $A$, satisfies the {\it max-flow min-cut\/}
(MFMC) 
property if both sides 
of the LP-duality equation
\begin{equation}\label{jun6-2-03-1}
{\rm min}\{\langle \alpha,x\rangle \vert\, x\geq 0; xA\geq{1}\}=
{\rm max}\{\langle y,{1}\rangle \vert\, y\geq 0; Ay\leq\alpha\} 
\end{equation}
have integral optimum solutions $x$ and $y$ for each non-negative
integral vector $\alpha$.  The system $x\geq 0; xA\geq {1}$ is called 
{\it totally dual integral\/}
(TDI) if the maximum in
Eq.~(\ref{jun6-2-03-1}) has an integral optimum solution $y$ for each
integral vector $\alpha$ with  
finite maximum.
\end{definition}

\begin{definition}\rm If $\alpha_0({\mathcal C})=\beta_1({\mathcal C})$ we say that the clutter 
$\mathcal C$ (or the ideal $I$) has the {\it K\"onig property\/}.
\end{definition}

Note that $\mathcal C$ has the packing property if and only if every
minor of $\mathcal C$ satisfies the K\"onig property. This leads to
the following well-known result.

\begin{corollary} \cite{cornu-book} \label{aug25-06-1}%\label{jan20-05} 
If a clutter
$\mathcal C$ has the max-flow
min-cut property,  
then $\mathcal C$ has the packing property. 
\end{corollary}

\begin{proof} Assume that the clutter $\mathcal{C}$ has the max-flow min-cut
property. This property is closed under taking minors.  Thus it 
suffices to prove that $\mathcal C$ has 
the K\"onig property. We denote the incidence matrix of ${\mathcal C}$ by
$A$. By hypothesis the LP-duality equation
$$
{\rm min}\{\langle{1},x\rangle \vert\, x\geq 0; xA\geq {1}\}=
{\rm max}\{\langle y,{1}\rangle \vert\, y\geq 0; Ay\leq{1}\}
$$
has optimum integral solutions $x$, $y$. To complete the 
proof notice that the left hand side of this 
equality is $\alpha_0({\mathcal C})$ and the right hand side 
is $\beta_1({\mathcal C})$. 
\end{proof}

Conforti and Cornu\'ejols \cite{CC} conjecture that the converse is also
true.% see also \cite[Conjecture~1.6]{cornu-book}. 

\begin{conjecture} {\rm (Conforti-Cornu\'ejols)} \label{conforti-cornuejols}\rm
If a clutter $\mathcal C$ has the
packing property, then $\mathcal C$ has the max-flow min-cut property.
\end{conjecture} 

An algebraic description of the packing property has already been
given. In order to use algebraic techniques to attack this
combinatorial conjecture, an algebraic translation is needed for the
max-flow min-cut property. There are several equivalent algebraic 
descriptions of the max-flow 
min-cut property, as seen in the following result.

\begin{theorem} \cite{normali,clutters,HuSV} \label{noclu1} Let
$\mathcal{C}$ be a clutter and let $I$ be its edge ideal. The
following conditions are equivalent{\rm :}
\begin{itemize}
\item[{\rm(i)\ \ }] ${\rm gr}_I(R)$ is reduced.
\item[{\rm(ii)\ }] $R[It]$ is normal and $Q(A)$ is an integral
polyhedron.
\item[{\rm(iii)}] $x\geq 0;\, xA\geq {1}$ is a {\rm TDI
system}.
\item[{\rm(iv)}] $\mathcal C$ has the max-flow min-cut property. 
\item[{\rm(v)}] $I^i=I^{(i)}$ for $i\geq 1$, where $I^{(i)}$ is the
$i${\it th\/} symbolic power. 
\item[{\rm(vi)}] $I$ is normally torsion-free.
\end{itemize}
\end{theorem}

By Theorems~\ref{noclu1} and \ref{lehman},
Conjecture~\ref{conforti-cornuejols} reduces to: 

\begin{conjecture} \cite{reesclu} \label{con-cor-vila}\rm If $I$ has
  the packing 
property, then $R[It]$ is normal.
\end{conjecture} 

Several variations of condition $\mbox{(ii)}$ of Theorem~\ref{noclu1} are
possible. In particular, there are combinatorial conditions on the
clutter that can be used to replace the normality of the Rees
algebra. One such condition is defined below. 

\begin{definition}\rm Let $\mathcal{C}^\vee$ be the clutter of minimal vertex 
covers of $\mathcal C$. The clutter $\mathcal C$ is called {\it
diadic\/} if 
$|e\cap e'|\leq 2$ for $e\in E({\mathcal C})$ and $e'\in E({\mathcal
C}^\vee)$
\end{definition}

\begin{proposition} \cite{reesclu} \label{nov2-03-1} If $Q(A)$ is
integral and $\mathcal C$ is diadic, then $I$ is normally
torsion-free.  
\end{proposition}

Theorem~\ref{noclu1} can be used to exhibit families of normally torsion-free
ideals. Recall that a matrix $A$ is 
called {\it totally unimodular\/} if each 
$i\times i$ subdeterminant of 
$A$ is $0$ or $\pm 1$ for all $i\geq 1$. 

\begin{corollary}\label{dec26-02} If $A$ is totally unimodular, then
$I$ and $I^\vee$ are normally torsion-free. 
\end{corollary}

\begin{proof}
By \cite{Schr2} the linear system $x\geq 0;\, xA\geq
{1}$ is TDI. Hence $I$ is normally torsion-free by
Theorem~\ref{noclu1}. Let $\mathcal{C}^\vee$ be the blocker (or
Alexander dual) of $\mathcal C$. By 
\cite[Corollary~83.1a(v), p.~1441]{Schr2}, we get that
$\mathcal{C}^\vee$ satisfies the max-flow min-cut
property. Hence $I(\mathcal{C}^\vee)$ is normally torsion-free by
Theorem~\ref{noclu1}. Thus $I^\vee$ is normally torsion-free because 
$I(\mathcal{C}^\vee)=I^\vee$. 
\end{proof}

In particular if $I$ is the edge ideal
of a bipartite graph, then $I$ and $I^\vee$ are normally torsion-free.

Theorem~\ref{noclu1} shows that the Rees algebra and the associated
graded ring play an important role in the study of the max-flow
min-cut property. An invariant related to the blowup algebras will
also be useful. 
The {\it analytic spread\/} of an edge ideal $I$ is given by 
$\ell(I)=\dim R[It]/\mathfrak{m}R[It]$. If $\mathcal{C}$ is uniform,
the analytic spread of $I$ is the rank of the incidence matrix of
$\mathcal{C}$. The analytic spread of a monomial ideal can be 
computed in terms of the Newton polyhedron of $I$, see
\cite{bivia-ausina}. The next result follows directly from
\cite[Theorem~3]{mcadam}. 

\begin{proposition}\label{march19-06} If $Q(A)$ is integral, then
$\ell(I)<n=\dim(R)$. 
\end{proposition}

To relate this result on $\ell(I)$ to
Conjecture~\ref{conforti-cornuejols} (or equivalently to
Conjecture~\ref{con-cor-vila}) we first need to recall the
following bound on the depths of the powers of an ideal $I$.

\begin{theorem}\label{burch} ${\rm inf}_i\{{\rm
depth}(R/I^i)\}\leq\dim(R)-\ell(I)$. If\/ ${\rm gr}_I(R)$ is
Cohen-Macaulay, then the equality holds. 
\end{theorem}

This inequality is due to Burch \cite{Bur} (cf.
\cite[Theorem~5.4.7]{huneke-swanson-book}), while the equality comes
from \cite{E-H}. By a result of Brodmann \cite{brodmann}, 
${\rm depth}\, R/I^k$ is constant for $k\gg 0$. Broadmann improved
Burch's inequality by showing that the constant value is bounded 
by $\dim(R)-\ell(I)$. For a study of the initial and limit behaviour
of the numerical 
function $f(k)={\rm depth}\, R/I^k$ see \cite{HH}. 

\begin{theorem} \cite{Hu0} \label{jul6-04} Let $R$ be a Cohen-Macaulay ring and
let $I$ be an ideal of $R$ containing regular elements. If $R[It]$ 
is Cohen-Macaulay, then ${\rm gr}_I(R)$ is Cohen-Macaulay.
\end{theorem}

\begin{proposition}\label{march18-06-1} 
 Let $\mathcal{C}$ be a clutter and let $I$ be its edge ideal. Let
 $J_i$ be the ideal obtained from $I$ by making $x_i=1$. If $Q(A)$ 
is integral, then $I$ is normal if and only
if $J_i$ 
is normal for all
$i$ and ${\rm depth}(R/I^k)\geq 1$ for all $k\geq 1$.
\end{proposition}

\begin{proof} Assume that $I$ is normal. The normality of an edge ideal is closed under
taking minors \cite{normali}, hence $J_i$ is normal
for all $i$. By hypothesis the Rees algebra $R[It]$ is normal. Then
$R[It]$ is Cohen-Macaulay by a theorem of Hochster \cite{Ho1}. 
Then the ring ${\rm gr}_I(R)$ is
Cohen-Macaulay by Theorem~\ref{jul6-04}. Hence 
using Theorem~\ref{burch} and Proposition~\ref{march19-06}  we get
that ${\rm depth}(R/I^i)\geq 1$ for all $i$. The converse 
follows readily adapting the arguments given in the proof 
of the normality criterion presented in \cite{normali}.
\end{proof}

By Proposition~\ref{march18-06-1} and Theorem~\ref{lehman}, we get that
Conjecture~\ref{conforti-cornuejols} also reduces to: 

\begin{conjecture}\rm If $I$ has the packing property, then 
${\rm depth}(R/I^i)\geq 1\ \mbox{ for all }i\geq 1$.
\end{conjecture}

We conclude this section with a collection of results giving
conditions under which Conjecture~\ref{conforti-cornuejols}, or its
equivalent statements mentioned above, is known to hold. 
For  uniform clutters it suffices to prove 
Conjecture~\ref{conforti-cornuejols}
for Cohen-Macaulay clutters \cite{cm-mfmc}.

\begin{proposition} \cite{reesclu} Let $\mathcal C$ be the collection of bases of a
matroid. If $\mathcal C$ satisfies the packing property,
then $\mathcal{C}$ satisfies the max-flow min-cut property. 
\end{proposition}

When $G$ is a graph, integrality of $Q(A)$ is sufficient in condition
$\mbox{(ii)}$ of Theorem~\ref{noclu1}, and the packing property is sufficient
to imply the max-flow min-cut property, thus providing another class
of examples for which Conjecture~\ref{conforti-cornuejols} holds.

\begin{proposition} \cite{cornu-book,reesclu} If $G$ is a graph and
$I=I(G)$, then the following 
are equivalent{\rm :}
\begin{itemize}
\item[\rm (a)] ${\rm gr}_I(R)$ is reduced.
\item[\rm (b)] $G$ is bipartite.
\item[\rm (c)] $Q(A)$ is integral.
\item[\rm (d)] $G$ has the packing property. 
\item[\rm (e)] $G$ has the max-flow min-cut property.
\item[\rm (f)] $\overline{I^i}=I^{(i)}$ for $i\geq 1$. 
\end{itemize}
\end{proposition}

\begin{definition}\rm A clutter is {\it binary\/} 
if its edges and its minimal
vertex covers intersect in an odd number of vertices.
\end{definition}

\begin{theorem}\cite{seymour} \label{binary-seymour} A binary clutter $\mathcal C$ has the 
max-flow min-cut property if and only if ${\mathcal Q}_6$ is not a minor 
of $\mathcal C$.
\end{theorem}

\begin{corollary} If $\mathcal C$ is a binary clutter with the packing
property, then $\mathcal C$ has the max-flow min-cut property. 
\end{corollary}

\begin{proposition} \cite{perfect} \label{perfect-polyhedra}
Let $\mathcal C$ be a uniform clutter and 
let $A$ be its incidence matrix. If the polyhedra 
$$
P(A)=\{x\vert\, x\geq 0;\,
xA\leq{1}\}\ \mbox{ and }\ Q(A)=\{x\vert\, x\geq 0;\,
xA\geq{1}\}
$$
are integral, then $\mathcal C$ has the max-flow min-cut property. 
\end{proposition}

In light of Theorem~\ref{lehman}, this result implies  
that if $P(A)$ is integral and $\mathcal{C}$ has the packing
property, then $\mathcal{C}$ has the max-flow min-cut property. An
open problem is to show that this result holds for non-uniform
clutters (see \cite[Conjecture~1.1]{chordal}).  

A {\it Meyniel graph\/} is a simple graph in which
every odd cycle of length at least five has at least two chords. The
following gives some support to \cite[Conjecture~1.1]{chordal} 
because Meyniel graphs are perfect \cite[Theorem~66.6]{Schr2}.

\begin{theorem} \cite{chordal} Let $\mathcal{C}$ be the clutter of maximal cliques of
a Meyniel graph. If $\mathcal{C}$ has the packing property, then
$\mathcal{C}$ has the max-flow min-cut property.
\end{theorem}

Let $P=(X,\prec)$ be a {\it partially ordered set\/} ({\it poset} for
short) on the finite vertex set $X$ 
and let $G$ be its {\it comparability graph\/}. Recall that the vertex 
set of $G$ is $X$ and the edge set of $G$ is the set of all unordered
pairs $\{x_i,x_j\}$ such that $x_i$ and $x_j$ are comparable. 

\begin{theorem} \cite{poset} \label{compara-ntf-normal} If $G$ is a comparability
graph and $\mathcal{C}$ is the clutter of maximal cliques of $G$, then
the edge ideal $I(\mathcal{C})$ is normally torsion free.   
\end{theorem}

\begin{theorem} \cite{mfmc} Let $\mathcal{C}$ be a uniform clutter
with a 
perfect matching such that 
$\mathcal{C}$ has the packing property and $\alpha_0(\mathcal{C})=2$.
If the columns of the incidence matrix of $\mathcal{C}$ are linearly
independent, then $\mathcal C$ has the max-flow min-cut property.
\end{theorem}

\section{ACKNOWLEDGMENTS}
The software package Macaulay 2
\cite{mac2} was used to compute many of the examples in this
paper, including Example~\ref{notchain}, which was originally
discovered by three undegraduate students (Mike Alwill, Katherine
Pavelek, and Jennifer von Reis) in an unpublished project in 2002. The
authors would like to thank an anonymous referee for providing us 
with useful comments and suggestions.

\bibliographystyle{plain}

\begin{thebibliography}{10}

\bibitem{barile} M. Barile, On the arithmetical rank of the edge ideals of
forests, Comm. Algebra {\bf 36} (2008), no. 12, 4678 - 4703.


\bibitem{berge1}{C. Berge, Some common properties for 
regularizable graphs, edge-critical graphs and $B$-graphs, 
{\it in} ``Theory and practice of combinatorics'' (A. Rosa, 
G. Sabidussi and J. Turgeon, Eds.), North-Holland 
Math. Stud. {\bf 60}, North-Holland, Amsterdam, 1982, 
pp.~31--44. }

\bibitem{bermejo-gimenez} I. Bermejo and P. Gimenez, Saturation and
Castelnuovo-Mumford regularity, J. Algebra {\bf 303} (2006), no. 2, 
592-617.

\bibitem{bivia-ausina} C. Bivi$\grave{\rm a}$-Ausina, The analytic spread of 
monomial ideals, Comm. Algebra {\bf 31} (2003), no. 7, 3487--3496.

\bibitem{bjorner-topological} A. Bj\"orner, Topological methods, 
Handbook of combinatorics, Vol. 1, 2,  1819--1872, Elsevier,
Amsterdam, 
1995. 

\bibitem{BW} A. Bj\"orner and M. Wachs, 
Shellable nonpure complexes and posets I, 
Trans. Amer. Math. Soc. {\bf 348} (1996), 1299--1327. 

\bibitem{Boll}{B. Bollob\'as, {\em Modern Graph Theory\/}, 
Graduate 
Texts in  Mathematics {\bf 184}, Springer-Verlag, 
New York, 1998.}

\bibitem{brodmann} M. Brodmann, 
The asymptotic nature of the analytic spread, 
Math. Proc. Cambridge Philos. Soc. {\bf 86} (1979), 35--39. 

\bibitem{brod}{M. Brodmann, Asymptotic stability
 of $\mbox{\rm Ass}(M/I^nM)$,
Proc. Amer. Math. Soc. {\bf 74} (1979), 16--18.}


\bibitem{BHer}{W. Bruns and J. Herzog, 
{\em Cohen-Macaulay Rings\/},  Cambridge University 
Press, Cambridge, Revised Edition, 1997.}

\bibitem{normaliz2} W. Bruns and B. Ichim, \textsc{Normaliz 2.0}, 
Computing
normalizations of affine 
semigroups 2008. Available from \newline 
{\tt http://www.math.uos.de/normaliz}.

\bibitem{Bur}{L. Burch, Codimension and analytic spread, Proc. Camb.
Phil. Soc. {\bf 72} (1972), 369-373.}

\bibitem{co-chordal} A. H. Busch, F. F. Dragan and R. Sritharan, New
  min-max theorems 
for weakly chordal and dually chordal graphs, Lecture Notes in
  Computer Science {\bf 6509}, 207--218, 2010, Springer Verlag.


\bibitem{AJ}{J. Chen, S. Morey and A. Sung, The stable set of
  associated primes of the ideal of a graph, Rocky Mountain
  J. Math. {\bf 32} (2002), 71--89.}

\bibitem{ConcaDeNegri}{A. Conca and E. De Negri, $M$-sequences, graph
  ideals, and ladder ideals of linear type, J. Algebra {\bf 211} (1999), 599--624.}

\bibitem{CC}{M. Conforti and G. Cornu\'ejols, 
Clutters that pack and the Max-Flow Min-Cut property: A conjecture, 
{\it The Fourth Bellairs Workshop on Combinatorial Optimization\/} 
(W. R. Pulleyblank, F. B. Shepherd, eds.), 1993.}

\bibitem{cornu-book}{G. Cornu\'ejols, {\it Combinatorial optimization: 
Packing and covering\/}, CBMS-NSF Regional Conference Series in
Applied  
Mathematics {\bf 74}, SIAM (2001).}

\bibitem{corso-nagel-tams} A. Corso and U. Nagel, Monomial and toric ideals
associated to ferrers graphs, Trans. Amer. Math. Soc. {\bf 361}
(2009), no. 3, 1371-1395. 

\bibitem{kia-manoj} K. Dalili, M. Kummini, Dependence of betti numbers
  on characteristic. Preprint, 2010, {\tt arXiv:1009.4243v1}.

\bibitem{kia-faridi-traves} K. Dalili, S. Faridi and W. Traves, The reconstruction
conjecture and edge ideals, Discrete Math. {\bf 308} (2008), 2002-2010.

\bibitem{Dirac}{G. A. Dirac, On rigid circuit graphs,
Abh. Math. Sem. Univ. Hamburg {\bf 38} (1961), 71--76.}

\bibitem{dochtermann} A. Dochtermann and A.  Engstrom, 
Algebraic properties of edge ideals via combinatorial
topology, Electron. J. Combin. {\bf 16} (2009), no. 2, R2.   



\bibitem{cm-mfmc} L. A. Dupont, E. Reyes and R. H. Villarreal, 
Cohen-Macaulay clutters with 
combinatorial optimization properties and parallelizations of 
normal edge ideals, S\~ao Paulo J.
Math. Sci. {\bf 3} (2009), no. 1, 61-75.

\bibitem{mfmc} L. A. Dupont and R. H. Villarreal, 
Algebraic and combinatorial properties of ideals and 
algebras of uniform clutters of TDI systems, J. Comb. Optim., 
to appear. 

\bibitem{poset} L. A. Dupont and R. H. Villarreal, 
Edge ideals of clique clutters of comparability graphs and the
normality of monomial ideals, Math. Scand. {\bf 106} (2010), no. 1, 88--98.

\bibitem{duval} A. M. Duval, 
Algebraic shifting and sequentially Cohen-Macaulay simplicial
complexes. Electron. J. Combin. {\bf 3} (1996), no. 1, Research Paper 21. 

\bibitem{ER}{J.~A. Eagon and V. Reiner, Resolutions of 
Stanley-Reisner rings and Alexander duality, J. 
Pure Appl. Algebra {\bf 130} (1998), 265--275.}

\bibitem{Eisen}{D. Eisenbud, {\it Commutative Algebra with a view
toward Algebraic Geometry\/}, Graduate
Texts in  Mathematics {\bf 150}, Springer-Verlag, 1995.}


\bibitem{eisenbud-syzygies} D. Eisenbud, {\it The geometry of syzygies: A
second course in commutative algebra and algebraic geometry},
Graduate Texts in Mathematics {\bf 229}, Springer, New York, 2005.

\bibitem{eisenbud-goto} D. Eisenbud and S. Goto, Linear free
resolutions and minimal multiplicity, 
J. Algebra {\bf 88} (1984), no. 1, 189--133. 

\bibitem{E-H}{D. Eisenbud and C. Huneke, Cohen-Macaulay Rees 
algebras and their specializations, J. Algebra {\bf 81} (1983),
202--224.} 

\bibitem{emtander} E. Emtander, 
A class of hypergraphs that generalizes chordal graphs,  Math. Scand.
{\bf 106} (2010),  no. 1, 50--66. 

\bibitem{terai-reg-ex} V. Ene,  O. Olteanu and N. Terai, Arithmetical rank of lexsegment
edge ideals, 
Bull. Math. Soc. Sci. Math. Roumanie (N.S.) {\bf 53}  (2010),
no. 4, 315--327.



\bibitem{erdos-gallai} P. Erd\"os and T. Gallai, On 
the minimal number of vertices representing 
the edges of a graph, Magyar Tud. Akad. Mat. Kutat\'o Int. 
K\"ozl. {\bf 6} (1961), 181--203.

\bibitem{normali} C. Escobar, R. H. Villarreal and Y. Yoshino, Torsion
freeness and normality of blowup rings of monomial ideals, 
{\it Commutative Algebra\/}, Lect. Notes Pure Appl. Math. 
{\bf 244}, Chapman \& Hall/CRC, Boca Raton, FL, 2006, pp. 69-84. 


\bibitem {EV}{M. Estrada and R. H. Villarreal, Cohen-Macaulay 
bipartite graphs,
Arch. Math. {\bf 68} (1997), 124--128.}

\bibitem{farber} M. Farber, Characterizations of strongly chordal
graphs, 
Discrete Math. {\bf 43} (1983), no. 2-3, 173–-189.

\bibitem{faridijct} S. Faridi, Cohen-Macaulay properties of
square-free  
monomial ideals, J. Combin. Theory Ser. A {\bf 109} (2005), no. 2, 
299--329.  

\bibitem{Faridi} S. Faridi, Simplicial trees are sequentially
Cohen-Macaulay, J. Pure Appl. Algebra {\bf 190} (2004), 121--136. 


\bibitem{FHVodd}{C. Francisco, H.T. H$\rm \grave{a}$, and A. Van Tuyl,
  Associated primes 
  of monomial ideals and odd holes in graphs, J. Algebraic
  Combin. {\bf 32} (2010), 287--301.}

\bibitem{FHV}{C. Francisco, H.T. H$\rm \grave{a}$ and A. Van Tuyl, Colorings of
  hypergraphs, perfect graphs, and associated primes of powers of
  monomial ideals, J. Algebra {\bf 331} (2011), no. 1,
224-242.}

\bibitem{FranciscoHaVanTuyl}{C. Francisco, H.T. H$\rm \grave{a}$ and 
A. Van Tuyl, A 
  conjecture on critical graphs and connections to the persistence of
  associated primes, Discrete Math. {\bf 310} (2010),
  2176--2182.} 

\bibitem{FVT2} C. A. Francisco and A. Van Tuyl, Sequentially
Cohen-Macaulay edge ideals,  Proc. Amer. Math. Soc.
{\bf 135} (2007), 2327--2337.


\bibitem{Fro4}{R. Fr\"oberg, On Stanley-Reisner rings, 
in {\em Topics in algebra\/} (S. Balcerzyk et. al., Eds.), 
Part 2. Polish Scientific Publishers,
1990, pp. 57--70.}


\bibitem{reesclu}{I. Gitler, E. Reyes and R. H. Villarreal, 
Blowup algebras of square--free monomial ideals and some links to
combinatorial optimization problems, 
Rocky Mountain J. Math. {\bf 39} (2009), no. 1, 71--102.} 

\bibitem{bounds} I. Gitler and C. Valencia, Bounds for 
invariants of edge-rings, Comm. Algebra {\bf 33} (2005), 1603--1616.   

\bibitem{clutters}{I. Gitler, C. Valencia and R. H. Villarreal, 
A note on Rees algebras and the MFMC property, Beitr\"age Algebra
Geom. {\bf 48} (2007), no. 1, 141--150.}

\bibitem{golumbic} M. C. Golumbic, Algorithmic graph theory and
perfect graphs, second edition, Annals of Discrete Mathematics  {\bf
57}, Elsevier Science B.V., Amsterdam, 2004.


\bibitem{mac2}{D.R. Grayson and M.E. Stillman, {\em Macaulay\/}$2$, 
a software system for research in algebraic geometry, 1996. 
\newblock {\tt http://www.math.uiuc.edu/Macaulay2/}.}

\bibitem{singular} G. M. Greuel and G. Pfister, {\it A Singular
introduction to commutative algebra}, Springer, Berlin, Second
Extended Edition, 2008.  

\bibitem{HaM} H. T. H$\rm \grave{a}$ and S. Morey, 
Embedded associated
primes of powers of square-free monomial ideals, 
J. Pure Appl. Algebra {\bf 214} (2010), no. 4, 301--308.

\bibitem{linearquotients} H. T. H\`{a}, S. Morey and R.H. Villarreal,
  Cohen-Macaulay 
admissible clutters,  J. Commut. Algebra {\bf 1} (2009), no. 3, 463--480.

\bibitem{Ha-VanTuyl} 
H. T. H$\rm \grave{a}$ and A. Van Tuyl, Monomial ideals,
edge ideals of hypergraphs, and their graded Betti numbers, J.
Algebraic Combin. {\bf 27} (2008), 215--245.

\bibitem{Har}{F. Harary, {\it Graph Theory\/}, Addison-Wesley, 
Reading, MA, 1972.}

\bibitem{He-VanTuyl} J. He and  A. Van Tuyl, Algebraic properties of the path
ideal of a tree, Comm. Algebra {\bf 38} (2010), no. 5, 1725-1742.

\bibitem{HIO}{M. Herrmann, S. Ikeda, U. Orbanz, {\it Equimultiplicity
    and blowing up. An algebraic study}, Springer-Verlag, Berlin,
  1988.} 

\bibitem{herzog-reg} 
J. Herzog, A Generalization of the Taylor complex construction, Comm. 
Algebra {\bf 35} (2007), no. 5, 1747--1756.

\bibitem{HeHi1}{J. Herzog and T. Hibi, Componentwise linear 
ideals, Nagoya Math. J. {\bf 153} (1999), 141--153.}

\bibitem{HH}{J. Herzog and T. Hibi,  The depth of powers of an
    ideal, J. Algebra {\bf 291} (2005), 534--550.} 

\bibitem{herzog-hibi} J. Herzog and T. Hibi, 
Distributive lattices, bipartite graphs and Alexander duality, 
J. Algebraic Combin. {\bf 22} (2005), no. 3, 289--302. 

\bibitem{herzog-hibi-book} J. Herzog and T. Hibi, {\it Monomial Ideals}, Graduate
Texts in  Mathematics {\bf 260}, Springer-Verlag, 2011.

\bibitem{HHTZ}{J. Herzog, T. Hibi, N.V. Trung and X. Zheng,
  Standard graded vertex cover algebras, cycles and
  leaves, Trans. Amer. Math. Soc. {\bf 360} (2008), 6231--6249.} 

\bibitem{hhz-ejc} J. Herzog, T. Hibi and X. Zheng, 
Dirac's theorem on chordal graphs and Alexander duality, 
European J. Combin. {\bf 25} (2004), no. 7, 949--960.

\bibitem{hhz-linear} J. Herzog, T. Hibi and X. Zheng, Monomial 
ideals whose powers have a linear resolution, Math. Scand. {\bf 95} 
(2004), 23--32.

\bibitem{Hoa}{L.T. Hoa, Stability of associated primes of monomial
  ideals, Vietnam J. Math. {\bf 34} (2006), no. 4, 473--487.}

\bibitem{Ho1}{M. Hochster, Rings of invariants of tori, 
{C}ohen-{M}acaulay rings generated by monomials, and polytopes, 
Ann. of Math. {\bf 96} (1972), 318--337.}


\bibitem{Hochster}{M. Hochster, Criteria for the equality of ordinary
and symbolic powers of primes, Math. Z. {\bf 133} (1973), 53--65.}

\bibitem{Hu0}{C. Huneke, On the associated graded ring of an ideal,
Illinois J. Math. {\bf 26} (1982), 121-137.}

\bibitem{HuSV}{C. Huneke, A. Simis and W. V. Vasconcelos, Reduced
normal cones are domains, Contemp.  Math. {\bf 88} (1989),
95--101.}

\bibitem{huneke-swanson-book} 
C. Huneke and I. Swanson, {\it Integral Closure of Ideals Rings, and
Modules}, London Math. Soc., Lecture Note Series {\bf 336}, Cambridge
University Press, Cambridge, 2006.

\bibitem{ionescu-rinaldo} C. Ionescu and G. Rinaldo, 
Some algebraic invariants related to mixed product ideals, 
Arch. Math. {\bf 91} (2008), 20-30.  

\bibitem{jacques} S. Jacques, Betti Numbers of Graph Ideals,
  Ph.D. Thesis, University of Sheffiels 2004, {\tt arXiv:math/0410107v1}. 

\bibitem{kalai-meshulam} G. Kalai and R. Meshulam,  Unions and
intersections of Leray complexes, J. Combin. Theory Ser. A {\bf 113} 
(2006), 1586--1592.


\bibitem{katzman1} M. Katzman, Characteristic-independence of Betti numbers of
graph ideals, J. Combin. Theory Ser. A {\bf 113} (2006), no. 3,
435--454.

\bibitem{kiani-moradi} D. Kiani and S. Moradi, Bounds for the
  regularity of edge ideal of vertex decomposable and shellable
  graphs. Preprint, 2010, {\tt arXiv:1007.4056v1}.

\bibitem{terai-yoshida-small} K. Kimura, N. Terai and K. Yoshida, Arithmetical rank of squarefree
monomial ideals of small arithmetic degree, J. Algebraic Combin. {\bf 29} (2009),
389-404.

\bibitem{kummini} M. Kummini, Regularity, depth and arithmetic rank of
bipartite edge ideals, J. Algebraic Combin. {\bf 30} (2009), no. 4,
429--445.  

\bibitem{lehman} A. Lehman, On the width-length inequality and degenerate 
projective planes, in {\it Polyhedral Combinatorics\/} (W. Cook and P.
Seymour Eds.) DIMACS Series in Discrete Mathematics and Theoretical
Computer Science {\bf 1}, Amer. Math. Soc., 1990, pp.~101-105 

\bibitem{lyubeznik-ara}
G. Lyubeznik, On the local cohomology modules $H_\mathcal{A}^i(R)$ for
ideals $\mathcal{A}$ generated by
monomials in an $R$-sequence. In: Greco, S., Strano, R., eds. Complete Intersections.
Lectures given at the 1st 1983 Session of the Centro Internazionale Matematico Estivo
(C.I.M.E.), Acireale, Italy, June 13-21, 1983; Berlin-Heidelberg:
Springer, pp. 214-220.

\bibitem{Lyu3}{G. Lyubeznik, On the arithmetical rank of 
monomial ideals, J. Algebra {\bf 112} (1988), 86--89.}

\bibitem{mmcrty}M. Mahmoudi, A. Mousivand, M. Crupi, G. Rinaldo,
  N. Terai, and S. Yassemi, Vertex decomposability and regularity of
  very well-covered graphs, to appear: J. Pure Appl. Algebra.  

\bibitem{chordal} J. Mart\'{\i}nez-Bernal, E. O'Shea and R. H. Villarreal, 
Ehrhart clutters: regularity and max-flow min-cut, Electron. J.
Combin. {\bf 17} (2010), no. 1, R52.   

\bibitem{mcadam}{S. McAdam, Asymptotic prime divisors and analytic
spreads, Proc. Amer. Math. Soc. {\bf 80} (1980), 555--559.} 

\bibitem{McAdam}{S. McAdam, {\em Asymptotic Prime Divisors}, Lecture
Notes in Mathematics {\bf 103}, Springer--Verlag, New York, 1983.}

\bibitem{ME}{S. McAdam and P. Eakin, The asymptotic ${\rm Ass}$, J.
Algebra {\bf 61} (1979), no. 1, 71--81.} 


\bibitem{cca}{ E. Miller and B. Sturmfels, {\it Combinatorial Commutative
Algebra\/}, Graduate Texts in  Mathematics {\bf 227}, Springer, 2004.}

\bibitem{morey} S. Morey, Depths of powers of the edge ideal of a
tree, Comm. Algebra {\bf 38} (2010), no. 11, 4042--4055.

\bibitem{morey2}{S. Morey, Stability of associated primes and equality
of ordinary and symbolic powers of ideals, Comm. Algebra {\bf 27}
(1999), 3221--3231.}


\bibitem{MRV} S. Morey, E. Reyes and R. H. Villarreal, 
Cohen-Macaulay, shellable and unmixed clutters with a perfect
matching of K\"{o}nig type, J. Pure Appl. 
Algebra {\bf 212} (2008), no. 7, 1770--1786.

\bibitem{nevo} E. Nevo, Regularity of edge ideals of $C_4$-free
graphs via the topology of the 
lcm-lattice, J. Combin. Theory Ser. A {\bf 118} (2011), no.2, 491--501.


\bibitem{peeva-stillman} I. Peeva and M. Stillman, Open problems on
syzygies and Hilbert functions, 
J. Commut. Algebra {\bf 1} (2009), no. 1, 159--195.

\bibitem{plummer-unmixed} M. D. Plummer, 
Some covering concepts in graphs, J. Combin. Theory {\bf 8}
(1970), 91--98. 

\bibitem{plummer-survey} M. D. Plummer, Well-covered graphs: a
survey, Quaestiones Math. {\bf 16} (1993), no. 3, 253--287.  

\bibitem{ravindra} G. Ravindra, Well-covered graphs, 
J. Combinatorics Information Syst. Sci. {\bf 2} (1977), no. 1, 20--21. 

\bibitem{renteln} P. Renteln, The Hilbert series of the face ring of
a flag complex, Graphs Combin. {\bf 18} (2002),  no. 3, 605--619. 

\bibitem{Schr}{A. Schrijver, {\it Theory of Linear and Integer
Programming\/}, John Wiley \& Sons, New York, 1986.}

\bibitem{Schr2} {A. Schrijver, {\it Combinatorial Optimization\/}, 
Algorithms and Combinatorics {\bf 24}, Springer-Verlag, Berlin, 2003.}

\bibitem{seymour} P.~D. Seymour, The matroids with the max-flow min-cut
property, J. Combin. Theory Ser. B {\bf 23} (1977), 189-222.

\bibitem{Sharp}{R. Sharp, Convergence of sequences of sets of
  associated primes. Proc. Amer. Math. Soc. {\bf 131} (2003), 3009--3017.} 

\bibitem{ITG}{A. Simis, W.~V. Vasconcelos and R. H. Villarreal, On 
the ideal theory of graphs, J. Algebra {\bf 167} 
(1994), 389--416.}

\bibitem{dsmith}{D.~E. Smith, On the Cohen-Macaulay property 
in commutative algebra and simplicial topology, Pacific J. Math. 
{\bf 141} (1990), 165--196.}

\bibitem{Stanley} R. P. Stanley, {\it Combinatorics and Commutative
Algebra. Second edition.}  
Progress in Mathematics {\bf 41}. Birkh{\"a}user Boston, Inc.,
Boston, MA, 1996. 


\bibitem{terai} N. Terai, Alexander duality theorem and
Stanley-Reisner rings, S\=urikaisekikenky\=usho K\=oky\=uroku {\bf
1078} (1999), 174--184, Free resolutions of coordinate rings 
of projective varieties and related topics (Kyoto 1998).   

\bibitem{Toft}
B. Toft, Colouring, stable sets and perfect graphs, in 
{\it Handbook of Combinatorics I\/} (R.~L.~Graham et. al., Eds.), 
 Elsevier, 1995, pp. 233--288.

\bibitem{Trung}{T. N. Trung, Stability of associated primes of
  integral closures of monomial ideals, J. Combin. Theory Ser. A {\bf 116}
  (2009),  44--54.}


\bibitem{vantuyl} A. Van Tuyl, Sequentially Cohen-Macaulay bipartite graphs:
vertex decomposability and 
regularity,  Arch. Math. {\bf 93} (2009), no. 5, 451--459. 

\bibitem{bipartite-scm} A. Van Tuyl and R. H. Villarreal, Shellable graphs and 
sequentially Cohen-Macaulay bipartite graphs, 
J. Combin. Theory Ser. A {\bf 115} (2008), no. 5, 799--814.

\bibitem{Vas1}{W. V. Vasconcelos, {\it Computational Methods in
Commutative Algebra and Algebraic Geometry\/}, 
Springer-Verlag, 1998.}

\bibitem{Vi2}{R. H. Villarreal, Cohen-{M}acaulay graphs, Manuscripta
Math. {\bf 66} (1990), 277--293.}

\bibitem{monalg}{R. H. Villarreal, {\it Monomial Algebras\/}, 
Monographs and 
Textbooks in Pure and Applied Mathematics {\bf 238}, Marcel 
Dekker, New York, 2001.} 

\bibitem{unmixed} R. H. Villarreal, Unmixed bipartite graphs, Rev. 
Colombiana Mat. {\bf 41} (2007), no. 2, 393--395. 

\bibitem{perfect}{R. H. Villarreal, Rees algebras and polyhedral cones
  of ideals of vertex covers of perfect graphs, J. Algebraic
  Combin. {\bf 27} (2008),  293--305.} 

\bibitem{Voloshin} V. I. Voloshin, {\it Coloring mixed hypergraphs: theory,
algorithms and applications}, Fields Institute Monographs {\bf 17}, 
American Mathematical Society, Providence, RI, 2002. 

\bibitem{whieldon} G. Whieldon, Jump sequences of edge
  ideals. Preprint, 2010, {\tt arXiv:1012.0108v1}.  

\bibitem{woodroofe-chordal} R. Woodroofe, Chordal and sequentially
  Cohen-Macaulay 
clutters. Preprint, 2009, {\tt arXiv:0911.4697v2}. 

\bibitem{Woodroofe} R. Woodroofe, Vertex decomposable graphs and obstructions 
to shellability, Proc. Amer. Math. Soc. {\bf 137} (2009), 3235--3246.

\bibitem{woodroofe-matchings} R. Woodroofe, Matchings, coverings, and
Castelnuovo-Mumford regularity. Preprint, 2010, {\tt arXiv:1009.2756}. 

\bibitem{zheng} X. Zheng, Resolutions of facet ideals, Comm. Algebra
{\bf 32} (2004), 2301--2324.

\end{thebibliography}

\end{document}